\documentclass[hidelinks,onefignum,onetabnum]{siamart250211}

\usepackage{lipsum}
\usepackage{amsfonts}
\usepackage{graphicx}
\usepackage{epstopdf}
\usepackage{algorithmic}
\ifpdf
  \DeclareGraphicsExtensions{.eps,.pdf,.png,.jpg}
\else
  \DeclareGraphicsExtensions{.eps}
\fi
\usepackage{tikz}
\usepackage{pgfplots}
\usepgfplotslibrary{groupplots}
\pgfplotsset{compat=1.18}

\newsiamremark{remark}{Remark}
\newsiamremark{hypothesis}{Hypothesis}
\crefname{hypothesis}{Hypothesis}{Hypotheses}
\newsiamthm{claim}{Claim}
\newsiamremark{fact}{Fact}
\crefname{fact}{Fact}{Facts}
\newsiamremark{assumption}{Assumption}

\headers{Parallel-in-time Newton's method-based ODE solver}{Casian Iacob, Hassan Razavi, Simo S\"arkk\"a}

\title{A parallel-in-time Newton's method-based ODE solver\thanks{Submitted to the editors.
\funding{
This work was funded by the Research Council of Finland and the Finnish Doctoral Program Network in Artificial Intelligence (AI-DOC).
}}}

\author{Casian Iacob\thanks{Department of Electrical Engineering and Automation, Aalto University, 02150 Espoo, Finland (\email{casian.iacob@aalto.fi}, \email{hassan.razavi@aalto.fi}, \email{simo.sarkka@aalto.fi})}
\and Hassan Razavi\footnotemark[2]
\and Simo S\"arkk\"a\footnotemark[2]}

\usepackage{amsopn}

\ifpdf
\hypersetup{
  pdftitle={A parallel-in-time Newton's method-based ODE solver},
  pdfauthor={Casian Iacob, Hassan Razavi, and Simo S\"arkk\"a}
}
\fi

\begin{document}

\maketitle

\begin{abstract}
In this article, we introduce a novel parallel-in-time solver for nonlinear ordinary differential equations (ODEs). We state the numerical solution of an ODE as a root-finding problem that we solve using Newton's method. The affine recursive operations arising in Newton's step are parallelized in time by using parallel prefix sums, that is, parallel scan operations, which leads to a logarithmic span complexity. This yields an improved runtime compared to the previously proposed Parareal method. We demonstrate the computational advantage through numerical simulations of various systems of ODEs.
\end{abstract}



\section{Introduction}
Numerous scientific and engineering applications involve the numerical solution of ordinary differential equations (ODEs). In many domains, solution methods must be fast due to real-time operation constraints or the computational demands of numerical solutions. This is the case, for example, in optimal control \cite{lewis2012optimal}, filtering and smoothing \cite{sarkka2023bayesian}, reinforcement learning \cite{sutton1998reinforcement}, and weather prediction \cite{lorenz2017deterministic}. One way to improve the speed of the solutions is to use parallelization, which has led to the development of various parallel-in-time solvers for ODEs \cite{gander201550}. 

There exist numerous methods to solve ODEs (see, e.g., \cite{butcher2016numerical, hairer1993solving, hairer1996solving}), and most of them rely on a sequential, iterative process involving a nonlinear function. Strategies of temporal parallelization have been developed to improve efficiency, which include the Parareal method \cite{lions2001resolution}, multigrid methods \cite{falgout2014parallel, dobrev2017two}, parallel spectral methods \cite{haut2014asymptotic, boyd2001chebyshev}, and exponential integrators \cite{gander2013paraexp}. A survey on parallel-in-time methods is available in \cite{gander201550}. Among these, the most popular approach is the Parareal algorithm \cite{lions2001resolution}. It has been heavily studied \cite{gander2007analysis, jin2025optimizing, gattiglio2024randnet, gattiglio2025prob, gander2025parareal} and extended to fit different problems, including optimal control \cite{gander2020paraopt, mathew2010analysis, maday2013parareal, fang2022parallel}, data assimilation \cite{bhatt2025introducing}, and stochastic differential equations \cite{legoll2020parareal, pentland2023error}. In Parareal \cite{lions2001resolution}, a cheap but inaccurate integrator is used to generate a rough solution which is then iteratively improved by using a more expensive but accurate integrator. A single-level implementation of Parareal can achieve square-root complexity. For multilevel Parareal \cite{rosemeier2024multilevel}, the computational span is given by the total critical path runtime. The coarse integrator steps are done sequentially, while the fine integrator only adds complexity through the required Parareal iterations.

In this paper, we leverage an emerging strategy for realizing temporal parallelization, which involves redefining recursive operations as generalized prefix-sum operations and using parallel associative scan \cite{blelloch1990prefix, blelloch1989scans} for computing them. A parallel associative scan is an algorithm that can be used to compute prefix sums for associative operations and associative elements in logarithmic time \cite{blelloch1989scans}. Implementations of this strategy already exist in the field of optimal control \cite{sarkka2024temporal, sarkka2022temporal, iacob2025parallel}, filtering and smoothing \cite{yaghoobi2025parallel, yaghoobi2024parallel}, dynamics computation in robotics \cite{yang2017parallel}, sorting algorithms \cite{satish2009designing}, and probabilistic ordinary and partial differential equation solvers \cite{bosch2024parallel, iqbal2024parallel}.  

Unfortunately, parallel associative scans are not directly applicable to parallelizing nonlinear ODE solvers, because they consist of arbitrary function compositions that, in general, do not correspond to associative operations on finite-dimensional elements. However, parallel associative scans are directly applicable to parallelizing affine ODE solutions \cite{blelloch1990prefix}, which is the property that we use in this paper. For this purpose, we define a sequence of affine problems that converge to the nonlinear solution of an ODE and thus can be used to realize the temporal parallelization of nonlinear ODE solvers. We achieve this by performing an explicit rollout of nonlinear equations obtained from applying a sequential ODE solver. This leads to a root-finding problem that we solve with Newton's method \cite{nocedal2006numerical}. We observe that the entries of the Newton step follow an affine recursive structure, which is parallelizable following the results from \cite{sarkka2022temporal}. It is worth noting that a bit similar ideas for parallelizing sequential operations using Newton's method with parallel cyclic reduction methods have previously appeared in machine learning literature \cite{gonzalez2024towards, gonzalez2026predictability, zoltowski2026parallelizing, lim2024parallelizing, danieli2023deeppcr}.

The remainder of this paper is structured as follows. In Section \ref{sec:background}, we introduce the problem and discuss standard and parallel numerical methods for solving it. Next, in Section \ref{sec:method}, we present the temporal parallelization of prefix-sum operations with focus on the affine difference equation. Next, we derive a parallel-in-time Newton step for realizing both the explicit and implicit integration of nonlinear differential equations. Theoretical results ensure the convergence of the derived Newton step. Finally, in Section \ref{sec:experiments}, we test and compare the average runtime needed for the integration of various ODEs against Parareal and standard sequential solvers. Moreover, we perform a numerical analysis on the convergence rate of our method.

\section{Background}\label{sec:background}
In this section, we formulate the initial value problem and provide a discussion on its solution. Furthermore, we describe the concept of temporal parallelization.

\subsection{Problem formulation}\label{sec:problem-formulation}
Consider the ordinary differential equation \begin{equation}
    \frac{d x}{dt} =f(x), \quad t\in[t_0, t_f], \quad x(t_0) = \Bar{x},\label{eq:ivp}
\end{equation}
where $x: [t_0, t_f] \rightarrow \mathbb{R}^{d_x}$ is the ODE solution,  $f:\mathbb{R}^{d_x} \rightarrow \mathbb{R}^{d_x}$ is a locally Lipschitz continuous function \cite{hairer1993solving}, and $\Bar{x}$ is the initial value at time $t_0$. Depending on the ODE function $f$, a solution to \eqref{eq:ivp} may be obtained analytically; however, this strategy is complex and only applies to limited special cases. Instead, a more practical approach is to approximate $x$ numerically \cite{butcher2016numerical, hairer1993solving}.

The numerical approximation of the ODE solution is typically formed \cite{butcher2016numerical, hairer1993solving} by first performing a temporal discretization of the integration interval $t_0, t_1, \hdots, t_N$ such that $\delta t = t_{n+1}-t_n$, $n=0, \hdots, N-1$. Then, at each discrete time step $t_n$, an approximate solution is computed by solving an explicit or implicit nonlinear equation. Examples of explicit methods are the explicit Euler's method
\begin{equation}
    x_{n+1} = x_n +f(x_n)\delta t, \quad n=0, \hdots, N-1, \label{eq:explicit-euler}
\end{equation}
and the family of explicit Runge--Kutta methods 
\begin{equation}
    \begin{split}
        x_{n+1} &= x_n +  \delta t \sum_{i=1}^{s}b_i \kappa_i, \quad n=0,\hdots, N-1,\\
        \kappa_i &= f\left(x_n + \delta t \sum_{j=1}^{i-1} a_{ij}\kappa_j\right), \quad i=1, \hdots, s,
    \end{split}\label{eq:explicit-rk}
\end{equation}
where the coefficients $a_{ij}$ and $b_i$ are selected from the corresponding Butcher tableau \cite{butcher2016numerical}, and $s$ is the number of stages.  We can generalize \eqref{eq:explicit-euler} and \eqref{eq:explicit-rk} to a nonlinear recursion of the form
\begin{equation}
    x_{n+1} = x_n + g_n(x_n, \delta t).
    \label{eq:explicit-ode-bg}
\end{equation}
Alternatively, one can also define an implicit numerical approximation. Examples of implicit methods are the implicit Euler's method 
\begin{equation}
    x_{n+1} = x_n + f(x_{n+1}) \delta t, \quad n=0,\hdots,N-1,\label{eq:implicit-euler}
\end{equation}
and the family of implicit Runge--Kutta methods
\begin{equation}
    \begin{split}
       x_{n+1} &=x_n +  \delta t \sum_{i=1}^{s}b_i \kappa_i, \quad n=0,\hdots,N-1,\\
        \kappa_i &= f\left(x_n + \delta t \sum_{j=1}^{s} a_{ij}\kappa_j\right), \quad i=1, \hdots, s,
    \end{split}\label{eq:implicit-rk}
\end{equation}
with parameters $a_{ij}$ and $b_i$ corresponding to a suitable Butcher tableau \cite{butcher2016numerical}. As in the explicit case, $s$ denotes the number of stages. Both the implicit Euler and implicit Runge--Kutta method, as well as many other implicit methods, correspond to an iteration of the form 
\begin{equation}
    x_{n+1} = x_n + g_n(x_n, x_{n+1}, \delta t) \label{eq:implicit-ode-bg}.
\end{equation} 
Eventually, the discrete sequence of ODE states $x_0, x_1, \hdots, x_N$, where $x_0=x(t_0)=\Bar{x}$, represents the approximate solution to \eqref{eq:ivp}. It is worth noting that although the implicit approach is essential for addressing stiff problems, it demands greater computational effort because a nonlinear root-finding problem must be solved at each time step.

\subsection{Parallel-in-time solution}
The numerical approximation methods \eqref{eq:explicit-ode-bg} and \eqref{eq:implicit-ode-bg} from Section \ref{sec:problem-formulation} have a recursive structure. The direct approach of computing each ODE state as a function of the previous state leads to a linear time complexity. For cases where the number of time steps is small, the sequential approach is sufficient. However, in many applications, such as nonlinear optimal control \cite{lewis2012optimal}, where the number of time steps is large and a nonlinear ODE must be solved repeatedly, reduced time complexity is desired. Temporal parallelization decomposes the time domain into multiple intervals and runs operations on the time intervals in parallel by removing recursive dependencies, reducing the linear time complexity.

The Parareal method~\cite{lions2001resolution} consists of solving \eqref{eq:ivp} by using a predictor-corrector approach where a coarse integrator is used for sequentially predicting the intermediate solution states, while the corrections are made by a fine integrator running in parallel. Let us define the coarse and fine integrators as $x(t+\delta t) \approx G(x(t), \delta t)$ and $x(t+\delta t) \approx F(x(t), \delta t)$, respectively. Next, let us introduce the temporal discretization of the integration interval by the intermediate time points $t_0, t_1, \hdots t_M$, where $\delta t = t_{m+1} - t_m$ and $M$ represents the number of time windows. Parareal consists of the following steps. We compute an initial sequence of ODE states using the coarse integrator
\begin{equation}
    x_{m+1}^{(0)} = G\left(x_m^{(0)}, \delta t\right), \quad x_0^{(0)} = \Bar{x},\label{eq:parareal-init}
\end{equation}
and denote $G_m^{(0)} = G\left(x_m^{(0)}, \delta t\right)$. Then, for $k=1, 2, \hdots, K$, we compute in parallel
\begin{equation}
    F_m^{(k-1)} = F\left(x_m^{(k-1)}, \delta t\right),\label{eq:fine-integrator}
\end{equation}
and perform the corrections sequentially
\begin{equation}
    \begin{split}
    G_m^{(k)} &= G\left(x_m^{(k)}, \delta t\right),\\
    x_{m+1}^{(k)} &= G_m^{(k)} + F_m^{(k-1)} - G_m^{(k-1)}.\label{eq:coarse-integrator}
    \end{split}
\end{equation}
After $K$ iterations, the dense trajectory of ODE states can be recovered by concatenating the intermediary states resulting from \eqref{eq:fine-integrator}. The number of time steps for the fine integrator is defined as $N/M$ where $N$ represents the dense trajectory obtained by sequentially concatenating the fine integrator steps. The purpose of this is to make the result of Parareal comparable with regular sequential methods as well as to our proposed method, which we introduce in the following section.

The parallel execution of \eqref{eq:fine-integrator} gives the computational advantage of the method; however, in \eqref{eq:parareal-init} and \eqref{eq:coarse-integrator}, we perform $M$ sequential steps in addition to $K$ overall iterations of the algorithm. Moreover, depending on the choice of discretization for the fine integrator, \eqref{eq:fine-integrator} also requires sequential operations. 

Let us now assume that we originally have $N \gg M$ steps and the coarse integrator $G$ uses a single step while the fine integrator $F$ uses $N/M$ steps. Then, the span complexity is $\mathcal{O}(K \, (M + N/M))$. If  $K < M$, then the span is $\mathcal{O}(M^2 + N)$ in the worst case. 
In the best case scenario, the optimal asymptotic complexity is given by $M = N^{1/2}$, which yields $\mathcal{O}(K \, N^{1/2})$, making the method of square-root complexity. A graphical representation of Parareal is provided in Figure \ref{fig:parareal-graphical}.

In this paper, we aim to introduce a different temporal parallelization strategy with improved parallelization that reduces the time complexity to logarithmic by exploiting parallel associative scans \cite{blelloch1989scans, blelloch1990prefix}.

\begin{figure}[t]
    \centering
    \begin{tikzpicture}[every node/.style={font=\scriptsize}]

\def\M{3} 
\def\Q{3}  

\foreach \i in {0,...,\M} {
  \ifnum\i=0
    \node[draw,minimum width=14mm,minimum height=5mm] (c\i) at (2.2*\i,1.6) {$x_0^{(k)}=\bar{x}$};
  \else
    \node[draw,minimum width=14mm,minimum height=5mm] (c\i) at (2.2*\i,1.6) {$x_{\i}^{(k)}$};
  \fi
}
\foreach \i in {0,...,\numexpr\M-1\relax} {
  \draw[->] (c\i) -- (c\the\numexpr\i+1\relax);
}

\foreach \i in {0,...,\numexpr\M-1\relax} {
  \path let \p1 = ($(c\i)!0.5!(c\the\numexpr\i+1\relax)$) in coordinate (m\i) at (\x1,1.0);
  \foreach \q in {1,...,\Q} {
    \ifnum\q=\Q
      \node[draw,minimum width=14mm,minimum height=5mm] (f\i-\q) at ($(m\i)+(0,-0.8*\q)$) {$F_{\i}^{(k-1)}$};
    \else
      \node[draw,minimum width=14mm,minimum height=5mm] (f\i-\q) at ($(m\i)+(0,-0.8*\q)$) {$x_{\i,\q}^{(k-1)}$};
    \fi
    \ifnum\q>1
      \draw[->] (f\i-\the\numexpr\q-1\relax) -- (f\i-\q);
    \fi
  }
  \draw[->] (c\i) -- (f\i-1);
  \draw[->] (f\i-\Q) -| (c\the\numexpr\i+1\relax);
}

\end{tikzpicture}
    \caption{Parareal dependency graph with $M=3$ windows. The coarse solver $G$ (horizontal top sequence) is single-step. The fine solver $F$ (vertical parallel sequences) takes $N=9$ steps in total ($N/M$ per window). The per-iteration span is $M + N/M$. Selecting $M=\sqrt{N}$ gives $\mathcal{O}(6)$ span complexity.}
    \label{fig:parareal-graphical}
\end{figure}

\section{Temporal parallelization of Newton's method for ODE integration} \label{sec:method}
In this section, we will review the temporal parallelization of prefix-sum operations and tailor it to affine difference equations. Then, we will define the explicit and implicit integration of an ODE as a nonlinear equation, which we solve via Newton's method. Following this strategy will enable us to realize the parallelization of nonlinear ODE solvers. Finally, we provide convergence results for the derived methods.

\subsection{Prefix-sum operations}\label{sec:temporal-parallelization-prefix-sum}
In the most favorable case, Parareal has square-root complexity. However,  a better, logarithmic span-complexity can be achieved by exploiting parallel associative scans \cite{blelloch1989scans, blelloch1990prefix} which are parallel algorithms for computing generalized prefix sums for associative operators. A prefix-sum is defined as follows. Consider an input sequence of elements 
\begin{equation}
    [a_0, a_1, \hdots, a_{N-1}],
\end{equation}
and an associative binary operator $\otimes$. The prefix-sum operation returns a sequence of equal length as the input sequence, containing the following elements
\begin{equation}
    [a_0, (a_0 \otimes a_1), \hdots, (a_0 \otimes a_1 \otimes \cdots \otimes a_{N-1})] = [a_0, a_{0, 1}, \hdots, a_{0,N-1}].\label{eq:prefix-sum}
\end{equation}
The sequential implementation of \eqref{eq:prefix-sum} has linear complexity $\mathcal{O}(N)$. We can reduce it to logarithmic by using a parallel associative scan as described in \cite{blelloch1989scans, blelloch1990prefix}. 

\begin{figure}[t]
    \centering
    \begin{tikzpicture}[font=\small]
    \def\W{12.0}
    \def\s{0.5}                               
    \pgfmathsetmacro{\g}{(\W-16*\s)/15} 
    \def\yleaf{0}
    \def\yone{1.35}
    \def\ytwo{2.70}
    \def\ythree{4.05}

    \foreach \i in {0,...,15} {
        \pgfmathsetmacro{\x}{\i*(\s+\g)}
        \draw (\x,\yleaf) rectangle ++(\s,\s);
    }

    \foreach \offset in {0,8} {

        \foreach \j in {0,1,2,3} {
            \pgfmathtruncatemacro{\a}{\offset + 2*\j}
            \pgfmathtruncatemacro{\b}{\offset + 2*\j + 1}

            \pgfmathsetmacro{\xa}{\a*(\s+\g) + 0.5*\s}
            \pgfmathsetmacro{\xb}{\b*(\s+\g) + 0.5*\s}
            \pgfmathsetmacro{\xp}{(\xa+\xb)/2}
            \pgfmathsetmacro{\xnode}{\xp - 0.5*\s}

            \draw (\xnode,\yone) rectangle ++(\s,\s);
            \draw (\xa,\yleaf+\s) -- (\xp,\yone);
            \draw (\xb,\yleaf+\s) -- (\xp,\yone);
        }

        \foreach \j in {0,1} {
            \pgfmathtruncatemacro{\a}{\offset + 4*\j}
            \pgfmathtruncatemacro{\b}{\offset + 4*\j + 1}
            \pgfmathtruncatemacro{\c}{\offset + 4*\j + 2}
            \pgfmathtruncatemacro{\d}{\offset + 4*\j + 3}

            \pgfmathsetmacro{\xa}{\a*(\s+\g) + 0.5*\s}
            \pgfmathsetmacro{\xb}{\b*(\s+\g) + 0.5*\s}
            \pgfmathsetmacro{\xc}{\c*(\s+\g) + 0.5*\s}
            \pgfmathsetmacro{\xd}{\d*(\s+\g) + 0.5*\s}

            \pgfmathsetmacro{\xleft}{(\xa+\xb)/2}
            \pgfmathsetmacro{\xright}{(\xc+\xd)/2}
            \pgfmathsetmacro{\xp}{(\xleft+\xright)/2}
            \pgfmathsetmacro{\xnode}{\xp - 0.5*\s}

            \draw (\xnode,\ytwo) rectangle ++(\s,\s);
            \draw (\xleft,\yone+\s) -- (\xp,\ytwo);
            \draw (\xright,\yone+\s) -- (\xp,\ytwo);
        }

        \pgfmathtruncatemacro{\a}{\offset}
        \pgfmathtruncatemacro{\b}{\offset+1}
        \pgfmathtruncatemacro{\c}{\offset+2}
        \pgfmathtruncatemacro{\d}{\offset+3}
        \pgfmathtruncatemacro{\e}{\offset+4}
        \pgfmathtruncatemacro{\f}{\offset+5}
        \pgfmathtruncatemacro{\h}{\offset+6}
        \pgfmathtruncatemacro{\k}{\offset+7}

        \pgfmathsetmacro{\xa}{\a*(\s+\g) + 0.5*\s}
        \pgfmathsetmacro{\xb}{\b*(\s+\g) + 0.5*\s}
        \pgfmathsetmacro{\xc}{\c*(\s+\g) + 0.5*\s}
        \pgfmathsetmacro{\xd}{\d*(\s+\g) + 0.5*\s}
        \pgfmathsetmacro{\xe}{\e*(\s+\g) + 0.5*\s}
        \pgfmathsetmacro{\xf}{\f*(\s+\g) + 0.5*\s}
        \pgfmathsetmacro{\xh}{\h*(\s+\g) + 0.5*\s}
        \pgfmathsetmacro{\xk}{\k*(\s+\g) + 0.5*\s}

        \pgfmathsetmacro{\xleft}{(((\xa+\xb)/2) + ((\xc+\xd)/2))/2}
        \pgfmathsetmacro{\xright}{(((\xe+\xf)/2) + ((\xh+\xk)/2))/2}
        \pgfmathsetmacro{\xroot}{(\xleft+\xright)/2}
        \pgfmathsetmacro{\xnode}{\xroot - 0.5*\s}

        \draw (\xnode,\ythree) rectangle ++(\s,\s);
        \draw (\xleft,\ytwo+\s) -- (\xroot,\ythree);
        \draw (\xright,\ytwo+\s) -- (\xroot,\ythree);
    }

    \foreach \i in {0,1,2,3,4,5,6,7} {
        \pgfmathsetmacro{\x}{\i*(\s+\g) + 0.5*\s}
        \node at (\x,\yleaf + 0.5*\s) {\scriptsize $a_{\i}$};
    }

    \node at ({0.5*((0*(\s+\g)+0.5*\s)+(1*(\s+\g)+0.5*\s))},\yone+0.5*\s) {\scriptsize $a_{0,1}$};
    \node at ({0.5*((2*(\s+\g)+0.5*\s)+(3*(\s+\g)+0.5*\s))},\yone+0.5*\s) {\scriptsize $a_{2,3}$};
    \node at ({0.5*((4*(\s+\g)+0.5*\s)+(5*(\s+\g)+0.5*\s))},\yone+0.5*\s) {\scriptsize $a_{4,5}$};
    \node at ({0.5*((6*(\s+\g)+0.5*\s)+(7*(\s+\g)+0.5*\s))},\yone+0.5*\s) {\scriptsize $a_{6,7}$};

    \node at ({0.5*(0.5*((0*(\s+\g)+0.5*\s)+(1*(\s+\g)+0.5*\s)) + 0.5*((2*(\s+\g)+0.5*\s)+(3*(\s+\g)+0.5*\s)))},\ytwo+0.5*\s) {\scriptsize $a_{0,3}$};
    \node at ({0.5*(0.5*((4*(\s+\g)+0.5*\s)+(5*(\s+\g)+0.5*\s)) + 0.5*((6*(\s+\g)+0.5*\s)+(7*(\s+\g)+0.5*\s)))},\ytwo+0.5*\s) {\scriptsize $a_{4,7}$};

    \pgfmathsetmacro{\xUpRoot}{0.5*(0.5*(0.5*((0*(\s+\g)+0.5*\s)+(1*(\s+\g)+0.5*\s)) + 0.5*((2*(\s+\g)+0.5*\s)+(3*(\s+\g)+0.5*\s))) + 0.5*(0.5*((4*(\s+\g)+0.5*\s)+(5*(\s+\g)+0.5*\s)) + 0.5*((6*(\s+\g)+0.5*\s)+(7*(\s+\g)+0.5*\s))))}
    \node at (\xUpRoot,\ythree+0.5*\s) {\scriptsize $a_{0,7}$};
    \node[right=2pt] at (\xUpRoot+0.25,\ythree+0.5*\s) {\scriptsize $=z_8$};

    \foreach \i/\txt in {
        8/{$\mathbf{e}$},
        9/{$a_0$},
        10/{$a_{0,1}$},
        11/{$a_{0,2}$},
        12/{$a_{0,3}$},
        13/{$a_{0,4}$},
        14/{$a_{0,5}$},
        15/{$a_{0,6}$}
    }{
        \pgfmathsetmacro{\x}{\i*(\s+\g) + 0.5*\s}
        \node at (\x,\yleaf + 0.5*\s) {\scriptsize \txt};
    }

    \foreach \i/\z in {
        8/{},
        9/{$=z_1$},
        10/{$=z_2$},
        11/{$=z_3$},
        12/{$=z_4$},
        13/{$=z_5$},
        14/{$=z_6$},
        15/{$=z_7$}
    }{
        \pgfmathsetmacro{\x}{\i*(\s+\g) + 0.5*\s}
        \node[below=2pt] at (\x,\yleaf) {\scriptsize \z};
    }

    \node at ({0.5*((8*(\s+\g)+0.5*\s)+(9*(\s+\g)+0.5*\s))},\yone+0.5*\s) {\scriptsize $\mathbf{e}$};
    \node at ({0.5*((10*(\s+\g)+0.5*\s)+(11*(\s+\g)+0.5*\s))},\yone+0.5*\s) {\scriptsize $a_{0,1}$};
    \node at ({0.5*((12*(\s+\g)+0.5*\s)+(13*(\s+\g)+0.5*\s))},\yone+0.5*\s) {\scriptsize $a_{0,3}$};
    \node at ({0.5*((14*(\s+\g)+0.5*\s)+(15*(\s+\g)+0.5*\s))},\yone+0.5*\s) {\scriptsize $a_{0,5}$};

    \node at ({0.5*(0.5*((8*(\s+\g)+0.5*\s)+(9*(\s+\g)+0.5*\s)) + 0.5*((10*(\s+\g)+0.5*\s)+(11*(\s+\g)+0.5*\s)))},\ytwo+0.5*\s) {\scriptsize $\mathbf{e}$};
    \node at ({0.5*(0.5*((12*(\s+\g)+0.5*\s)+(13*(\s+\g)+0.5*\s)) + 0.5*((14*(\s+\g)+0.5*\s)+(15*(\s+\g)+0.5*\s)))},\ytwo+0.5*\s) {\scriptsize $a_{0,3}$};

    \node at ({0.5*(0.5*(0.5*((8*(\s+\g)+0.5*\s)+(9*(\s+\g)+0.5*\s)) + 0.5*((10*(\s+\g)+0.5*\s)+(11*(\s+\g)+0.5*\s))) + 0.5*(0.5*((12*(\s+\g)+0.5*\s)+(13*(\s+\g)+0.5*\s)) + 0.5*((14*(\s+\g)+0.5*\s)+(15*(\s+\g)+0.5*\s))))},\ythree+0.5*\s) {\scriptsize $\mathbf{e}$};
\end{tikzpicture}
    \vspace{-3em}
    \caption{Parallel scan for all prefix sums of an $N=8$ input sequence where the associative operation represents the function composition and the elements are affine functions parameterized as in \eqref{eq:assoc-operator}. The up-sweep (left) starts from the leaves $[a_0, \hdots, a_7]$, and applies the composition operation on each level, yielding the final state $z_8 = a_0 \otimes \cdots \otimes a_7$ at the root. The down-sweep (right) starts from the root by initializing it with the identity function $\mathbf{e}(x) = x$. Then on each level, the root is copied to the left child, while the right child gets the result of applying the function composition between the root and its left child from the up-sweep tree. This yields the intermediary prefix sums at the leaves of the down-sweep tree. Performing the up-sweep and down-sweep has a complexity span of $\mathcal{O}(\log_2(8)) = \mathcal{O}(3)$ respectively. Therefore, the computation of all prefix sums via an associative scan has a span of $\mathcal{O}(6)$.} 
    \label{fig:up-and-down-sweep}
\end{figure}

\begin{figure}[t]
    \centering
    \begin{tikzpicture}[font=\small]
    \def\W{12.0}
    \def\s{0.75}
    \pgfmathsetmacro{\g}{(\W-8*\s)/7}

    \def\ybox{0}

    \foreach \i in {0,...,7} {
        \pgfmathsetmacro{\x}{\i*(\s+\g)}
        \draw (\x,\ybox) rectangle ++(\s,\s);
    }

    \foreach \i in {0,...,6} {
        \pgfmathtruncatemacro{\qind}{\i+1}

        \pgfmathsetmacro{\xleft}{\i*(\s+\g)+\s}
        \pgfmathsetmacro{\xright}{(\i+1)*(\s+\g)}
        \pgfmathsetmacro{\xmid}{0.5*(\xleft+\xright)}
        \pgfmathsetmacro{\ymid}{\ybox+0.5*\s}

        \draw[->] (\xleft,\ymid) -- (\xmid-0.13,\ymid);

        \node[inner sep=1pt] (op\i) at (\xmid,\ymid) {\scriptsize $\otimes$};

        \draw[->] (\xmid+0.13,\ymid) -- (\xright,\ymid);

        \node[above=18pt, inner sep=0pt] (q\i) at (\xmid,\ymid) {\scriptsize $a_{\qind}$};

        \draw[->] (q\i.south) -- (op\i.north);
    }

    \node at ({0*(\s+\g)+0.5*\s},\ybox+0.5*\s) {\scriptsize $a_0$};
    \node at ({1*(\s+\g)+0.5*\s},\ybox+0.5*\s) {\scriptsize $a_{0,1}$};
    \node at ({2*(\s+\g)+0.5*\s},\ybox+0.5*\s) {\scriptsize $a_{0,2}$};
    \node at ({3*(\s+\g)+0.5*\s},\ybox+0.5*\s) {\scriptsize $a_{0,3}$};
    \node at ({4*(\s+\g)+0.5*\s},\ybox+0.5*\s) {\scriptsize $a_{0,4}$};
    \node at ({5*(\s+\g)+0.5*\s},\ybox+0.5*\s) {\scriptsize $a_{0,5}$};
    \node at ({6*(\s+\g)+0.5*\s},\ybox+0.5*\s) {\scriptsize $a_{0,6}$};
    \node at ({7*(\s+\g)+0.5*\s},\ybox+0.5*\s) {\scriptsize $a_{0,7}$};

    \node[below=2pt, inner sep=0pt] at ({0*(\s+\g)+0.5*\s},\ybox) {\scriptsize $=z_1$};
    \node[below=2pt, inner sep=0pt] at ({1*(\s+\g)+0.5*\s},\ybox) {\scriptsize $=z_2$};
    \node[below=2pt, inner sep=0pt] at ({2*(\s+\g)+0.5*\s},\ybox) {\scriptsize $=z_3$};
    \node[below=2pt, inner sep=0pt] at ({3*(\s+\g)+0.5*\s},\ybox) {\scriptsize $=z_4$};
    \node[below=2pt, inner sep=0pt] at ({4*(\s+\g)+0.5*\s},\ybox) {\scriptsize $=z_5$};
    \node[below=2pt, inner sep=0pt] at ({5*(\s+\g)+0.5*\s},\ybox) {\scriptsize $=z_6$};
    \node[below=2pt, inner sep=0pt] at ({6*(\s+\g)+0.5*\s},\ybox) {\scriptsize $=z_7$};
    \node[below=2pt, inner sep=0pt] at ({7*(\s+\g)+0.5*\s},\ybox) {\scriptsize $=z_8$};

    \path[use as bounding box] (0,-0.45) rectangle (\W,1.75);
\end{tikzpicture}
    \vspace{-3em}
    \caption{Sequential computation of all prefix sums of an $N=8$ input sequence where the associative
    operation represents the function composition and the elements are affine functions as defined in
    (3.3). The computation of all prefix sums has a span of $\mathcal{O}(8)$.} 
    \label{fig:sequential}
\end{figure}

\begin{algorithm}[t]
\caption{Parallel scan \cite{blelloch1990prefix}}
\label{alg:par-scan}
\begin{algorithmic}[1]
\REQUIRE Associative elements $[a_0, a_1, \hdots, a_{N-1}]$ and associative operator $\otimes$
\ENSURE All prefix sums $[a_0, (a_0 \otimes a_1), \hdots, (a_0\otimes a_1 \otimes \cdots \otimes a_{N-1})]$
\STATE $b \gets \mathrm{copy}(a)$ \COMMENT{Store the input}
\FOR[Up-sweep]{$n \gets 0$ \textbf{to} $\log_2N-1$}
\FOR[Compute in parallel]{$i \gets 0$ \textbf{to} $N-1$ \textbf{by} $2^{d+1}$}
    \STATE $j\gets i + 2^d$
    \STATE $n\gets i + 2^{d+1}$
    \STATE $a_n \gets a_j \otimes a_n$
\ENDFOR
\ENDFOR
\STATE $a_N \gets \mathbf{e}$ \COMMENT{The identity element $\mathbf{e}$ represents the identity function $\mathbf{e}(x)=x$}
\FOR[Down-sweep]{$d\gets\log_2N-1$ \textbf{to} $0$}
\FOR[Compute in parallel]{$i\gets0$ \textbf{to} $N-1$ \textbf{by} $2^{d+1}$}
\STATE $j\gets i+2^d$
\STATE $k\gets i+2^{d+1}$
\STATE $t\gets a_j$
\STATE $a_j \gets a_n$
\STATE $a_n \gets a_n \otimes t$
\ENDFOR
\ENDFOR
\FOR[Final pass to form the inclusive scan; Compute in parallel]{$n\gets 1$ \textbf{to} $N$}
\STATE $a_n \gets a_n \otimes b_n$
\ENDFOR
\end{algorithmic}
\end{algorithm}

In the following, we will introduce the solution to an affine difference equation as a prefix-sum operation and derive its temporal parallelization based on the results of \cite{sarkka2022temporal}. Consider the difference equation
\begin{equation}
    z_{n+1} = q_n(z_n) = F_n z_n + c_n, \quad n = \{0, 1, \hdots, N-1\}, \quad z_0 = \Bar{z},\label{eq:linear-model}
\end{equation}
where $z\in\mathbb{R}^{d_z}$ is the state, $F \in \mathbb{R}^{d_z\times d_z}$, and $c \in \mathbb{R}^{d_z}$. We can solve the difference equation \eqref{eq:linear-model} from an arbitrary initial state $z_n$ to a final state $z_i$ by performing the following function compositions
\begin{equation}
    \begin{split}
         z_i &= (q_{i-1} \circ \cdots \circ q_n)(z_n), \quad i>n.\\
         &= F_{n, i} z_n + c_{n,i}
    \end{split}
\end{equation}
Let us use the tuple $(F_{n,i}, c_{n_i})$ that parameterizes the resulting function composition to denote an associative element $a_{n,i}$. The combination of two elements $a_{j, i}$ and $a_{n, j}$ by the operator $\otimes$ is defined as
\begin{equation}
    \begin{split}
        a_{n, i} &= a_{j,i} \otimes a_{n,j}, \quad n < j < i,\\
        (F_{n, i}, c_{n, i}) &= (F_{j, i}, c_{j,i}) \otimes  (F_{n, j}, c_{n,j})\\
        &= (F_{j,i}F_{n,j}, F_{j,i}c_{n,j}+c_{j,i}).
    \end{split}\label{eq:assoc-operator}
\end{equation}
\begin{lemma}
    The binary operation $\otimes$ defined in \eqref{eq:assoc-operator} is associative.
\end{lemma}
\begin{proof}
    Let us consider three elements $a_{n,l}, a_{l, j}, a_{j, i}$, $n<l<j<i$. We will show that the following equality holds
    \begin{equation}
        (a_{j, i} \otimes a_{l, j}) \otimes a_{n,l} = a_{j, i} \otimes (a_{l, j} \otimes a_{n,l}).
    \end{equation}
    Starting with the left-hand side, we have
    \begin{equation}
        \begin{split}
            &[(F_{j, i}, c_{j, i}) \otimes (F_{l, j}, c_{l,j})] \otimes (F_{n,l}, c_{n,l}) = (F_{j, i}F_{l, j}, F_{j, i}c_{l, j} + c_{j, i}) \otimes (F_{n,l}, c_{n,l})\\
            &= (F_{l, i}, c_{l, i}) \otimes (F_{n,l}, c_{n,l}) = (F_{l, i} F_{n, l}, F_{l,i}c_{n, l} + c_{l,i}) = (F_{n, i}, c_{n, i}).
        \end{split}
    \end{equation}
    Continuing with the right-hand side
    \begin{equation}
        \begin{split}
            &(F_{j, i}, c_{j, i}) \otimes [(F_{l, j}, c_{l,j}) \otimes (F_{n,l}, c_{n,l})] = (F_{j, i}, c_{j, i}) \otimes (F_{l,j}F_{n, l}, F_{l, j}c_{n,l} + c_{l, j})\\
            &= (F_{j, i}, c_{j, i}) \otimes (F_{n,j}, c_{n,j}) = (F_{j, i} F_{n, j}, F_{j,i}c_{n, j} + c_{j,i}) = (F_{n, i}, c_{n, i}),
        \end{split}
    \end{equation}
    which proves the claim.
\end{proof}
The associative elements in \eqref{eq:prefix-sum} are then initialized as
\begin{equation}
    \begin{split}
        a_n &= (F_{n, n+1}, c_{n, n+1}), \quad F_{n, n+1} = F_n, \quad c_{n, n+1} = c_n, \quad 1 \leq n \leq  N-1,\\
        a_0 &= (F_{0,1}, c_{0,1}) = (0,  F_0 z_0 + c_0).
    \end{split}
    \label{eq:initialization}
\end{equation}
 By performing a parallel associative scan, we can compute the states $z_{1:N}$ in $\mathcal{O}(\mathrm{log}N)$-time
\begin{equation}
    z_n = a_0 \otimes a_1 \otimes \cdots \otimes a_{n-1}, \quad 0\leq n \leq N.
\end{equation}
The implementation details of the parallel associative scan are presented in Algorithm \ref{alg:par-scan}. However, the algorithm is already implemented in various software packages, for example, JAX \cite{jax2018github}.

An illustrative example of the parallel scan for computing all prefix sums of $N=8$ elements is shown in Figure \ref{fig:up-and-down-sweep}. For comparison, we provide a diagram for the same computation executed sequentially in Figure \ref{fig:sequential}.

We observe that the result from \ref{eq:assoc-operator} does not extend to the nonlinear case because the composition of nonlinear functions does not correspond to associative operations on finite-dimensional elements. Therefore, the derived strategy is limited to affine recursions. To overcome this hurdle, we formulate the equations resulting from the numerical approximation of ODE solutions as a system of nonlinear equations and solve it via Newton's method. At each iteration, the resulting Newton step has an affine recursive structure, and thus it can be computed in parallel with an associative scan as presented in this section.

\subsection{Parallel-in-time explicit ODE solvers}\label{sec:pint-explicit-ode}
Consider the initial value problem \eqref{eq:ivp}, and recall the explicit ODE solution approximation methods \eqref{eq:explicit-euler} and \eqref{eq:explicit-rk} introduced in Section \ref{sec:problem-formulation}, and the nonlinear recursion resulting from their generalization 
\begin{equation}
    x_{n+1} = x_n + g_n(x_n, \delta t), \quad x_0 = \Bar{x}, \label{eq:explicit-ode}
\end{equation}
where $\delta t$ is the discretization step. 
If we explicitly write the iterations resulting from \eqref{eq:explicit-ode}, we obtain the following nonlinear equation
\begin{equation}
    h(\xi)  = 
    \begin{bmatrix}
        h_1\\ h_2 \\ \vdots\\ h_N
    \end{bmatrix} = \begin{bmatrix}
        x_1 - x_0 - g_0(x_0, \delta t)\\
        x_2 - x_1 - g_1(x_1, \delta t)\\
        \vdots\\
        x_N - x_{N-1} - g_{N-1}(x_{N-1}, \delta t)
    \end{bmatrix}= 0,\label{eq:root-finding-problem}
\end{equation}
where $\xi^\top = \begin{bmatrix}x_1 ^\top & \hdots & x_N^\top\end{bmatrix}, \, \xi \in \mathbb{R}^{d_xN}$, is the concatenated vector of all intermediary ODE states. Problem \eqref{eq:root-finding-problem} poses a root-finding problem that we solve with Newton's method \cite{nocedal2006numerical}. The Jacobian $\partial h / \partial \xi = H(\xi)$ has the following entries 
\begin{equation}
    H_{ij} = \frac{\partial h_i}{\partial x_j} = \begin{cases}I, & i = j,\\
    -I - \frac{\partial g_j}{\partial x_j}, & j = i - 1,\\
    0, & \text{otherwise},\end{cases}\label{eq:jacobian}
\end{equation}
where $I \in \mathbb{R}^{d_x \times d_x}$ represents the identity matrix. Given an initial guess $\xi^{(0)}$, Newton's method consists of the following iterations
\begin{align}
    \Delta \xi^{(k)} &\gets -H^{-1}\left(\xi^{(k)}\right) h\left(\xi^{(k)}\right),\label{eq:newton-step}\\
    \xi^{(k+1)} &\gets \xi^{(k)} + \Delta \xi^{(k)},\label{eq:newton-iterates}
\end{align}
where $\Delta \xi^\top = \begin{bmatrix}\Delta x_1^\top&\hdots&\Delta x_N^\top\end{bmatrix}$ is the Newton step. Given the lower-bidiagonal structure of \eqref{eq:jacobian}, we can expand \eqref{eq:newton-step} as follows
\begin{equation}
\begin{split}
\Delta x_1 &= -H_{11}^{-1} h_1, \\
\Delta x_2 &= -H_{22}^{-1}(h_2 + H_{21} \Delta x_1), \\
\vdots \\
\Delta x_N &= -H_{NN}^{-1}(h_N + H_{NN-1} \Delta x_{N-1}).\label{eq:newton-step-explicit}
\end{split}
\end{equation}
Note that $H_{ii} = I, \, \forall i=1, \hdots, N$. Hence, the scheme in \eqref{eq:newton-step-explicit} does not require any matrix inversion. Substituting \eqref{eq:jacobian} in \eqref{eq:newton-step-explicit}, we obtain
\begin{equation}
    \Delta x_n = \left(I + \frac{\partial g_{n-1}}{\partial x_{n-1}}\right) \Delta x_{n-1} - h_n, \quad n = 2, \hdots, N, \quad \Delta x_1 = -h_1.\label{eq:newton-step-recursion}
\end{equation}
We observe that \eqref{eq:newton-step-recursion} has the same affine structure as \eqref{eq:linear-model} and, therefore, we can realize its temporal parallelization following the steps presented in Section \ref{sec:temporal-parallelization-prefix-sum}. Furthermore, one can see the straightforward connection of the derived Newton step and the multiple-shooting strategy \cite[Chap.\,2]{gander2024time}\cite{gander2007analysis}. In this article, the concatenated vector of ODE states $\xi$ has the same meaning as the shooting parameters in \cite[Chap.\,2]{gander2024time}. In the next section, we provide a convergence analysis of the parallel Newton step for the explicit ODE numerical solution.

\subsection{Theoretical results for the parallel-explicit Newton's method} In this section, we prove that Newton's method for explicit ODE numerical solution converges at a quadratic rate, by roughly following the general proof of Newton's method from \cite{nocedal2006numerical}. To ensure the convergence of the proposed method, we require the following assumption.

\begin{assumption}\label{assum2}
    The Jacobian $H$ of the function $h$, is Lipschitz continuous in the neighborhood of $\xi^{*}$, that is, there exists $L>0$ such that
    \begin{equation}
        \|H(\xi)-H(\xi^{*})\| \leq L \|\xi-\xi^{*}\|,
    \end{equation}
    where $h\left(\xi^*\right)=0$.
\end{assumption}

\begin{theorem}\label{theorem1}
    Suppose Assumption \ref{assum2} holds. Let the sequence of Newton iterates $\left\{\xi^{(k)}\right\}$ be defined by \eqref{eq:newton-iterates}, and let the Newton step $\Delta \xi^{(k)}$ be defined by \eqref{eq:newton-step}. If the initial guess $\xi^{(0)}$ is sufficiently close to $\xi^*$, the sequence $\{\xi^{(k)}\}$ converges quadratically to $\xi^*$ and the sequence $\left\{\left\|h\left(\xi^{(k)}\right)\right\|\right\}$ converges quadratically to zero.
\end{theorem}
\begin{proof}
The complete mathematical proof is given in Appendix
\ref{ap:convergence_explicit_ODEs}. In the following, we summarize the main steps. To prove the quadratic convergence of the proposed parallel Newton step, we roughly follow the convergence analysis of Newton's method while tailoring the bounds to the specific affine recursive structure of our explicit ODE solver. By leveraging the Lipschitz continuity of the Jacobian $H$ (Assumption \ref{assum2}) , we can establish an initial quadratic bound on the iteration error $\left\|\xi^{(k+1)}-\xi^{*}\right\|$. The critical point in our specific setting here is bounding the inverse of the Jacobian, $\left\|H^{-1}(\xi^{k})\right\|$, during the iterative process. Since $H$ in \eqref{eq:jacobian} is a block lower bidiagonal matrix with identity blocks on the main diagonal, it is inherently invertible at the exact solution $\xi^*$. Consequently, by applying a Neumann series expansion, we guarantee that the inverse remains bounded. This implies that the matrix inversion does not amplify errors, thereby preserving the classical quadratic convergence rate for both the sequential iterates and the residual norms.
\end{proof}
A numerical analysis of the proven convergence results is presented in Section \ref{sec:explicit-experiments}. Next, let us derive a parallel-in-time implicit ODE solver and analyze its convergence by following the same strategy as we did in the explicit case. 

\subsection{Parallel-in-time implicit ODE solvers} \label{sec:pint-implicit-ode}
Recall the implicit ODE solution approximation methods \eqref{eq:implicit-euler} and \eqref{eq:implicit-rk} presented in Section \ref{sec:problem-formulation} as well as their generalization to
\begin{equation}
    x_{n+1} = x_n + g_n(x_n, x_{n+1}, \delta t), \quad x_0 = \Bar{x} \label{eq:implicit-ode}.
\end{equation}
Examples here include backward Euler \eqref{eq:implicit-euler} or the trapezoidal rule
\begin{equation}
    g(x_n, x_{n+1}, \delta t) = f(x_{n}) \frac{\delta t}{2} + f(x_{n+1}) \frac{\delta t}{2}.
\end{equation}
The explicit rollout of \eqref{eq:implicit-ode} yields the following nonlinear equation
\begin{equation}
    h(\xi)  = 
    \begin{bmatrix}
        h_1\\ h_2 \\ \vdots\\ h_N
    \end{bmatrix} = \begin{bmatrix}
        x_1 - x_0 - g_0(x_0, x_1, \delta t)\\
        x_2 - x_1 - g_1(x_1, x_2, \delta t)\\
        \vdots\\
        x_N - x_{N-1} - g_{N-1}(x_{N-1}, x_{N}\delta t)
    \end{bmatrix}= 0,\label{eq:root-finding-problem-implicit}
\end{equation}
where $\xi^\top = \begin{bmatrix}x_1 ^\top & \hdots & x_N^\top\end{bmatrix}, \, \xi \in \mathbb{R}^{d_xN}$ is the concatenated vector of intermediary ODE states. We deploy the same strategy for solving the root-finding problem \eqref{eq:root-finding-problem-implicit} as in Section \ref{sec:pint-explicit-ode}. The Jacobian $\partial h/\partial \xi = H$ has the following lower-bidiagonal structure 
\begin{equation}
    H_{ij} = \frac{\partial h_i}{\partial x_j} = \begin{cases}I - \frac{\partial g_{i-1}}{\partial x_i}, & i = j,\\
    -I - \frac{\partial g_j}{\partial x_j}, & j = i - 1,\\
    0, & \text{otherwise},\end{cases}\label{eq:jacobian-implicit}
\end{equation}
and, therefore, the explicit iterations for the Newton step entries are identical to \eqref{eq:newton-step-explicit}. For convenience, we write the Newton iterations again
\begin{align}
    \Delta \xi^{(k)} &\gets -H^{-1}\left(\xi^{(k)}\right) h\left(\xi^{(k)}\right),\label{eq:newton-step-implicit}\\
    \xi^{(k+1)} &\gets \xi^{(k)} + \Delta \xi^{(k)}.\label{eq:newton-iterates-implicit}
\end{align}
Substituting \eqref{eq:jacobian-implicit}, we obtain the following affine recursion
\begin{equation}
    \begin{split}
        \Delta x_n &= -\left(I - \frac{\partial g_{n-1}}{\partial x_n}\right)^{-1}\left[\left(-I - \frac{\partial g_{n-1}}{\partial x_{n-1}}\right) \Delta x_{n-1} + h_n\right],
    \\
    \Delta x_1 &= -\left(I - \frac{\partial g_0}{\partial x_1}\right)^{-1} h_1.\label{eq:newton-step-implicit-recursion}
    \end{split}
\end{equation}
We can realize the temporal parallelization of \eqref{eq:newton-step-implicit-recursion} following the steps presented in Section \ref{sec:temporal-parallelization-prefix-sum}. In the next Section, we provide a convergence analysis of the parallel Newton step for the implicit ODE numerical solution. 

\subsection{Theoretical results for the parallel-implicit Newton's method}
The implicit method introduced in Section \ref{sec:pint-explicit-ode} enjoys the same quadratic convergence results as the explicit method. Let us prove them next.
\begin{assumption}\label{assum3}
    The determinants 
    \begin{equation}
\det\left(\frac{\partial g(x_{n-1}^*,x_{n}^*,\delta t)}{\partial x_n^*}-I\right) \neq 0, \quad n = 1, \hdots, N,
\end{equation}
where $(\xi^*)^\top = \begin{bmatrix}(x_1^*)^\top & \hdots & (x_N^*)^\top\end{bmatrix}$, and $h(\xi^*)=0$. 
\end{assumption}
\begin{theorem}
Suppose that Assumptions \ref{assum2} and \ref{assum3} are satisfied for the Jacobian $H$ defined in \eqref{eq:jacobian-implicit}.
Let the Newton iterates $\{\xi^{(k)}\}$ be generated by \eqref{eq:newton-iterates-implicit} with Newton steps $\Delta \xi^{(k)}$ given in \eqref{eq:newton-step-implicit}.
If the initial point $\xi^{(0)}$ is sufficiently close to the solution $\xi^*$, then the sequence $\{\xi^{(k)}\}$ converges to $\xi^*$ with quadratic rate.
Moreover, the residuals $\|h(\xi^{(k)})\|$ decay to zero quadratically.
\end{theorem}
\begin{proof}
   The only difference between the implicit method and the explicit one lies in the
Jacobian matrix $H$ defined in \eqref{eq:jacobian-implicit}.  
Hence, for the implicit ODE integrator, it suffices to examine the
invertibility of $H$.  

Since $H$ in \eqref{eq:jacobian-implicit} is a block lower bidiagonal matrix, its eigenvalues depend on the diagonal blocks only. Let us rewrite the matrix \eqref{eq:jacobian-implicit} as
\begin{equation}
    H=A+B=A(I+A^{-1}B),
\end{equation}
where $A$ and $B$ are block diagonal and block subdiagonal matrices as
\begin{equation}
    \begin{aligned}
        A_{ij} = \begin{cases}I - \frac{\partial g_{i-1}}{\partial x_i}, & i = j,\\
0, & \text{otherwise.}\end{cases}, \quad \quad B_{ij} = \begin{cases}-I - \frac{\partial g_j}{\partial x_j}, & j = i - 1,\\
0, & \text{otherwise.}\end{cases}.
    \end{aligned}
\end{equation}
The determinant of $H$ can be rewritten as
\begin{equation}\label{detproperty}
    \det(H)=\det(A)\det(I+A^{-1}B),
\end{equation}
where $A^{-1}B$ is a nilpotent matrix since it is a strictly lower–triangular block matrix. This implies that \cite{bernstein2009matrix}
\begin{equation}
    \det(I+A^{-1}B)=1.
\end{equation}
Subsequently, according to \eqref{detproperty}, the determinant of $H$ is
\begin{equation}
    \det(H)=\det(A),
\end{equation}
which means
\begin{equation}
        \det(H)=\det(A)=\prod_{n=1}^{N} \det\left(\frac{\partial g_{n-1}}{\partial x_n}-I\right).
\end{equation}
By invoking Assumption \ref{assum3}, we guarantee that this product is strictly non-zero. Consequently, $H$ is inherently invertible.
 Given this characterization, the remainder of the proof is identical to that of Theorem~\ref{theorem1}.
\end{proof}
The proven convergence results for the implicit method are analyzed numerically in Section \ref{sec:implicit-experiments}.
 
\section{Numerical results}\label{sec:experiments}
In this section, we implement our proposed methods from Section \ref{sec:method} using the JAX software package \cite{jax2018github}. We test and compare both the explicit and implicit solvers from Sections \ref{sec:pint-explicit-ode} and \ref{sec:pint-implicit-ode}, respectively,  against Parareal \cite{lions2001resolution} and the sequential approach. All experiments are run on an NVIDIA Tesla A100, 80-gigabyte graphics processing unit (GPU). 

To evaluate the performance of our parallel-in-time method, we consider the following IVPs. For the explicit case, we solve the logistic equation \cite{tronarp2019probabilistic}, the van der Pol oscillator \cite{buonomo1998periodic}, and the unactuated cart-pole \cite{underactuated}. For the implicit case, we consider the Dahlquist test problem \cite{corless2019optimal} and the Robertson chemical reaction \cite{robertson1966reaction}. The listed test problems represent standard ODEs or systems of ODEs for benchmarking. The logistic equation and the Dahlquist test problem are simple, scalar problems with known solutions. The van der Pol system introduces nonlinear oscillatory dynamics. The cart-pole is a coupled nonlinear system, common for testing control algorithms, and the Robertson reaction system is a classic stiff ODE benchmark suitable for testing implicit methods.

In all cases, we fix the integration interval $[t_0, t_f]$ and vary the discretization step, such that as $\delta t$ decreases, the number of intermediary ODE states $N = (t_f-t_0)/\delta t$ increases. For Parareal, we perform M single-step coarse updates, where $M=N^{1/2}$. The fine integrator is set to perform $N/M$ sequential steps with the same $\delta t$ as in the sequential and parallel-Newton cases. We compare our implicit method to a sequential implementation, where an Optimistix Newton solver solves the nonlinear implicit equations \cite{rader2024optimistix}. The same implicit solver, featuring the Optimistix Newton solver, is used for the implicit Parareal solver. For both explicit and implicit experiments, we average the running time over 10 instances for each discretization step. 

To ensure a fair comparison between our proposed method and the existing Parareal, we run both methods until they reach the same solution accuracy. Therefore, we consider the following iteration stopping criteria. For our parallel-in-time Newton's method-based solver, we evaluate the infinity norm of the residual $\|h\left(\xi^{(k)}\right)\|_\infty = \sup\left|h(\xi^{(k)})\right|$. For Parareal, we define $\left\|R^{(k)}\right\|_\infty = \max_{0\leq m\leq M-1}\left|x_{m+1}^{(k)} - F_m^{(k)}\right|$. Further details can be found in the linked code base\footnote{https://github.com/casiacob/parallel-ode}.

\subsection{Explicit ODE solution approximation}\label{sec:explicit-experiments}
We consider a set of non-stiff initial value problems. The logistic equation defines the first IVP governed by the following ODE
\begin{equation}
    \frac{dP}{dt} = r\,P\,\left(1-\frac{P}{K}\right), \quad t \in [0, 10], \quad P(0) = 0.1,\label{eq:logistic}
\end{equation}
where $r=1$ and $K=1$.

Next, let us consider the van der Pol oscillator
\begin{equation}
    \frac{d^2x}{dt^2} = \mu\,\left(1-x^2\right)\frac{dx}{dt} - x, \quad t\in[0, 10], \quad x(0) = 0, \quad \frac{dx(0)}{dt} = 1,\label{eq:vdp}
\end{equation}
where $\mu=1$.

Finally, we study a cart-pole system 
\begin{equation}
    \begin{split}
    \frac{d^2p}{dt^2} &= \frac{m_p\,\sin\theta\,\left(l \frac{d\theta}{dt} + g \cos \theta\right)}{m_c + m_p\, \sin^2 \theta},\\
    \frac{d^2 \theta}{dt^2} &= \frac{-m_p\,l\,\omega^2\,\cos\theta\,\sin\theta - (m_c+m_p)\,g\,\sin\theta}{l\,(m_c+m_p\,\sin^2\theta)},\\
     t&\in[0, 4], \quad p(0) = 0, \quad \frac{dp(0)}{dt}=0, \quad \theta(0) = \frac{\pi}{2}, \quad \frac{d\theta(0)}{dt} = 0.
    \end{split}\label{eq:cartpole}
\end{equation}
The parameters in \eqref{eq:cartpole} are set to $g=9.81$, $l=0.5$, $m_c=10$, and $m_p=1$. 

We solve the IVPs \eqref{eq:logistic}, \eqref{eq:vdp}, and \eqref{eq:cartpole}, with the following time step sizes $\delta t \in \{10^{-2}, 10^{-3}, 10^{-4}, 10^{-5}\}$. For all methods, that is, the parallel Newton, Parareal (coarse and fine solver), and the sequential method, we used the same fourth-order Runge--Kutta rule \cite{butcher2016numerical}. The initial guess $\xi^{(0)}$ for the parallel Newton method is set to a vector of ones for \eqref{eq:logistic} and \eqref{eq:vdp}. For \eqref{eq:cartpole}, we initialize $\xi^{(0)}$ with a vector of zeros. The average resulting runtimes are displayed in Figure \ref{fig:explicit-runtime}. Compared to both Parareal and the sequential method, our approach yields a faster runtime and features improved scalability as the discrete integration horizon increases.

\begin{figure}[t]
    \centering
    \begin{tikzpicture}

    \begin{groupplot}[
        group style={
            group name=my plots,
            group size=3 by 1, 
            horizontal sep=5pt
        }, 
        width=5.0cm, 
        height=5.0cm, 
        grid=both,
        minor x tick num=3,
        grid style={line width=.1pt, draw=gray!20},
        major grid style={line width=.1pt, draw=gray!60},
        ytick distance=10^1,
        ytick={0.0001, 0.001, 0.01, 0.1, 1.0, 10.0},
        yticklabels={,$10^{-3}$,,$10^{-1}$,,$10^{1}$},
        ymin=0.0003, ymax=20,
        ylabel shift=-0.25cm,
    ]
    \nextgroupplot[
        title={Logistic},
        xmode=log,
        ymode=log,
        xlabel={ODE States $N$},
        ylabel={Runtime [s]},
        xtick align=inside,
        ytick align=inside,
        legend entries={Parallel Newton;,Parareal;,Sequential},
        legend to name=common_legend,
        legend style={legend columns=-1},
    ]
        \addplot [
            semithick, 
            black, 
            mark=o, 
            mark size=2, 
            mark options={solid}
        ] table {%
            1000 0.00071223
            10000 0.00091093
            100000 0.00129282
            1000000 0.00331168

        };
        \addplot [
            semithick, 
            black, 
            mark=triangle, 
            mark size=2, 
            mark options={solid}
        ] table {%
            1000 0.00288682
            10000 0.00614488
            100000 0.01883779
            1000000 0.05735576
        };

        \addplot [
            semithick, 
            black, 
            mark=square, 
            mark size=2, 
            mark options={solid}
        ] table {%
            1000 0.00924201
            10000 0.08497362
            100000 0.83502538
            1000000 8.27522111
        };
    \nextgroupplot[
        title={van der Pol},
        xmode=log,
        ymode=log,
        xlabel={ODE States $N$},
        ylabel={},
        yticklabel={\empty},
        xtick align=inside,
        ytick align=inside,
        ymin=0.0003, ymax=20,
    ]
        \addplot [
            semithick, 
            black, 
            mark=o, 
            mark size=2, 
            mark options={solid}
        ] table {%
            1000 0.00232096
            10000 0.00326538
            100000 0.00630546
            1000000 0.0343277
        };
        \addplot [
            semithick, 
            black, 
            mark=triangle, 
            mark size=2, 
            mark options={solid}
        ] table {%
            1000 0.00766256
            10000 0.02673819
            100000 0.13406012
            1000000 0.53979595
        };
        \addplot [
            semithick, 
            black, 
            mark=square, 
            mark size=2, 
            mark options={solid}
        ] table {%
            1000 0.01248844
            10000 0.12495108
            100000 0.96730559
            1000000 9.93280284
        };

            \nextgroupplot[
        title={Cart-pole},
        xmode=log,
        ymode=log,
        xlabel={ODE States $N$},
        ylabel={},
        yticklabel={\empty},
        xtick align=inside,
        ytick align=inside,
        ymin=0.0003, ymax=20,
    ]
        \addplot [
            semithick, 
            black, 
            mark=o, 
            mark size=2, 
            mark options={solid}
        ] table {%
            400 0.00254242
            4001 0.00345271
            40000 0.00485833
            400000 0.03394353

        };
        \addplot [
            semithick, 
            black, 
            mark=triangle, 
            mark size=2, 
            mark options={solid}
        ] table {%
            400 0.00818293
            4001 0.01184125
            40000 0.02902684
            400000 0.06637533
        };
        \addplot [
            semithick, 
            black, 
            mark=square, 
            mark size=2, 
            mark options={solid}
        ] table {%
            400 0.01006086
            4001 0.08939776
            40000 0.89697938
            400000 8.79305544
        };
    
\end{groupplot}
\node at ($(current bounding box.south)+(0,-0.5cm)$) {\pgfplotslegendfromname{common_legend}};
\end{tikzpicture}
    \caption{Average runtime comparison of explicit parallel Newton-based method, Parareal, and sequential integration for the logistic equation, the van der Pol oscillator, and the cart-pole system over a fixed time interval with varying numbers of intermediary steps.}
    \label{fig:explicit-runtime}
\end{figure}
 \begin{figure}[t]
    \centering
    \begin{tikzpicture}

    \begin{groupplot}[
        group style={
            group name=my plots,
            group size=3 by 1, 
            horizontal sep=5pt
        }, 
        width=5.0cm, 
        height=5.0cm, 
        grid=both,
        minor x tick num=3,
        grid style={line width=.1pt, draw=gray!20},
        major grid style={line width=.1pt, draw=gray!60},
        ytick distance=10^1,
        ytick={1e-16, 1e-12, 1e-8, 1e-4, 1},
        ymin=1e-17, ymax=5,
        ylabel shift=-0.25cm,
    ]
    \nextgroupplot[
        title={Logistic},
        ymode=log,
        xlabel={Iteration $k$},
        ylabel={},
        xtick align=inside,
        ytick align=inside,
        legend entries={$\left\|h\left(\xi^{(k)}\right)\right\|_\infty$;, $\left\|R^{(k)}\right\|_\infty$},
        legend to name=explicit_residual_legend,
        legend style={legend columns=-1},
    ]
        \addplot [
            semithick, 
            black, 
            mark=o, 
            mark size=2, 
            mark options={solid}
        ] table {%
            0 0.899096393102499
            1 0.008035310323891902
            2 0.0012283800521106355
            3 0.00029730924926136973
            4 1.5692007800689224e-05
            5 2.601641020194903e-08
            6 4.852888924045118e-14
            7 1.0866077460557066e-16
        };

        \addplot [
            semithick, 
            black, 
            mark=triangle, 
            mark size=2, 
            mark options={solid}
        ] table {%
            0 1.6946804630713075e-06
            1 1.183808606697312e-10
            2 8.43769498715119e-15
            3 4.440892098500626e-16
        };
    \nextgroupplot[
        title={van der Pol},
        ymode=log,
        xlabel={Iteration $k$},
        ylabel={},
        yticklabel={\empty},
        xtick align=inside,
        ytick align=inside,
        ymin=1e-17, ymax=5,
    ]
        \addplot [
            semithick, 
            black, 
            mark=o, 
            mark size=2, 
            mark options={solid}
        ] table {%
            0 0.9899500012604583
            1 0.02985150864451224
            2 0.9735425210729862
            3 0.26403451834297903
            4 0.22036455491333706
            5 0.12304066067657315
            6 0.002187954149964668
            7 5.836482835919199e-07
            8 5.614397835529417e-13
            9 4.2327252813834093e-16
        };

        \addplot [
            semithick, 
            black, 
            mark=triangle, 
            mark size=2, 
            mark options={solid}
        ] table {%
            0 0.00761938963492903
            1 0.00021649117542958674
            2 3.873812919968911e-06
            3 1.1232184547738111e-07
            4 9.58577550669304e-10
            5 6.437128607927889e-12
            6 7.599476603559197e-14
            7 8.881784197001252e-16
        };
            \nextgroupplot[
        title={Cart-pole},
        ymode=log,
        xlabel={Iteration $k$},
        ylabel={},
        yticklabel={\empty},
        xtick align=inside,
        ytick align=inside,
        ymin=1e-17, ymax=5,
    ]
        \addplot [
            semithick, 
            black, 
            mark=o, 
            mark size=2, 
            mark options={solid}
        ] table {%
            0 1.5698153268199289
            1 0.14268895131812304
            2 0.1432059516274663
            3 0.3719954952957657
            4 0.1117837869228559
            5 0.040514165122703866
            6 0.003683857913785643
            7 1.1861656544487342e-06
            8 1.379452108096757e-14
        };
         \addplot [
            semithick, 
            black, 
            mark=triangle, 
            mark size=2, 
            mark options={solid}
        ] table {%
            0 0.05968881652402924
            1 0.03917136764330209
            2 0.0004320384719438408
            3  1.899288608253613e-05
            4 1.4840786155545516e-07
            5 3.3016878120406545e-10
            6 3.843148022042442e-12
            7 3.1086244689504383e-13
        };
\end{groupplot}
\node at ($(current bounding box.south)+(0,-0.5cm)$) {\ref{explicit_residual_legend}};
\end{tikzpicture}
    \caption{Evolution of the residual norms over Newton and Parareal iterations for the logistic equation, van der Pol oscillator, and cart-pole.}
    \label{fig:explicit-residuals}
\end{figure}

\begin{figure}[t]
    \centering
    \begin{tikzpicture}

    \begin{groupplot}[
        group style={
            group name=my plots,
            group size=3 by 1, 
            horizontal sep=5pt
        }, 
        width=5.0cm, 
        height=5.0cm, 
        grid=both,
        minor x tick num=3,
        grid style={line width=.1pt, draw=gray!20},
        major grid style={line width=.1pt, draw=gray!60},
        ytick={2, 4, 6, 8, 10, 12, 14, 16, 18, 20},
        yticklabels={2, , , 8, , , 14, , , 20},
        ymin=1, ymax=21,
        ylabel shift=-0.25cm,
    ]
    \nextgroupplot[
        title={Logistic},
        xmode=log,
        xlabel={$\delta t$},
        ylabel={\empty},
        xtick align=inside,
        ytick align=inside,
        legend entries={Parallel Newton $K$;,Parareal $K$},
        legend to name=explicit_iterations_legend,
        legend style={legend columns=-1},
    ]
        \addplot [
            semithick, 
            black, 
            mark=o, 
            mark size=2, 
            mark options={solid}
        ] table {%
            0.01 8
            0.001 8
            0.0001 7
            1e-05 7
        };
        \addplot [
            semithick, 
            black, 
            mark=triangle, 
            mark size=2, 
            mark options={solid}
        ] table {%
            0.01 4
            0.001 3
            0.0001 3
            1e-05 3
        };
    \nextgroupplot[
        title={van der Pol},
        xmode=log,
        xlabel={$\delta t$},
        yticklabel={\empty},
        xtick align=inside,
        ytick align=inside,
    ]
        \addplot [
            semithick, 
            black, 
            mark=o, 
            mark size=2, 
            mark options={solid}
        ] table {%
            0.01 10
            0.001 10
            0.0001 10
            1e-05 9
        };
        \addplot [
            semithick, 
            black, 
            mark=triangle, 
            mark size=2, 
            mark options={solid}
        ] table {%
            0.01 8
            0.001 9
            0.0001 16
            1e-05 18
        };

            \nextgroupplot[
        title={Cart-pole},
        xmode=log,
        xlabel={$\delta t$},
        yticklabel={\empty},
        xtick align=inside,
        ytick align=inside,
    ]
        \addplot [
            semithick, 
            black, 
            mark=o, 
            mark size=2, 
            mark options={solid}
        ] table {%
            0.01 9
            0.001 9
            0.0001 9
            1e-05 9
        };
        \addplot [
            semithick, 
            black, 
            mark=triangle, 
            mark size=2, 
            mark options={solid}
        ] table {%
            0.01 8
            0.001 4
            0.0001 3
            1e-05 2
        };

\end{groupplot}
\node at ($(current bounding box.south)+(0,-0.5cm)$) {\pgfplotslegendfromname{explicit_iterations_legend}};
\end{tikzpicture}
    \caption{Influence of the time-step $\delta t$ on the convergence iteration number $K$ for the logistic equation, van der Pol oscillator, and cart-pole.}
    \label{fig:explicit-iterations}
\end{figure}

To analyze the convergence speed of Newton's method, we solve again the three IVPs, namely \eqref{eq:logistic}, \eqref{eq:vdp}, and \eqref{eq:cartpole}, with $\delta t = 10^{-2}$, and display the infinity norm of the residual $\left\|h\left(\xi^{(k)}\right)\right\|_\infty$ at each iteration, in Figure \ref{fig:explicit-residuals}. We compare the result against the error of the Parareal iterations $\left\|R^{(k)}\right\|_\infty$. In all three experiments, we can clearly observe that Parareal requires fewer iterations to reach the same accuracy as the Newton's method-based solver. Especially in the case of the van der Pol system, as well as the cart-pole system, the error in Newton's method only starts to decrease after the first 5 to 6 iterations. This behaviour is expected based on the local convergence analysis of the method, provided in \ref{theorem1}. Nevertheless, the prefix scan implementation of the Newton's method-based solver exhibits faster runtimes, showing that the method reaches a predefined solution accuracy in less time.

Finally, we also investigate the influence of varying the time step period on the number of iterations required for convergence. We display the results in Figure \ref{fig:explicit-iterations}. Similar to what we have already seen in Figure \ref{fig:explicit-residuals}, in most cases, Parareal converges within fewer iterations for the same level of solution accuracy.

\subsection{Implicit ODE solution approximation}\label{sec:implicit-experiments} 
To analyze the implicit version of our proposed method, we repeat the scenarios presented in Section \ref{sec:explicit-experiments}. The Dahlquist test defines the first IVP governed by the following ODE
\begin{equation}
    \frac{dy}{dt} = \lambda \, y, \quad t\in[0, 4],\quad y(0) = 1, \label{eq:dahlquist}
\end{equation}
where we set the parameter $\lambda = -1$. We solve \eqref{eq:dahlquist} with the following discretization periods $\delta t \in \{10^{-1}, 10^{-2}, 10^{-3}, 10^{-4}\}$.

Next, let us define Robertson's ODE for chemical reactions 
\begin{equation}
    \begin{split}
        \frac{d y_1}{dt} &= -k_1 \, y_1 + k_3 \, y_2 \, y_3,\\
        \frac{d y_2}{dt} &= k_1 y_1 - k_2 y_2^2 - k_3 \, y_2 \, y_3,\\
        \frac{d y_3}{dt} &= k_2 \, y_2^2,\\
        t&\in[0, 500],\quad y_1(0) = 1, \quad y_2(0) = 0, \quad y_3(0)=0,
    \end{split}\label{eq:robertson}
\end{equation}
where the parameters are set to $k_1 = 0.04$, $k_2=3\times10^7$, and $k_3=10^4$. We solve \eqref{eq:robertson} with the following periods: $\delta t \in \{10^{-1}, 10^{-2}, 5\times10^{-3}\}$.

For both systems and methods, we used the backward Euler implicit rule \cite{butcher2016numerical}. Moreover, we set the initial guess for our parallel-in-time Newton-based method to a vector of zeros, for both \eqref{eq:dahlquist} and \eqref{eq:robertson}. The results are presented in Figure \ref{fig:implicit-runtime}. Similar to the previous section, we observe that our method yields a considerably faster time than the sequential approach as well as the Parareal implementation, with favorable scaling as the number of intermediary states increases.
\begin{figure}[t]
    \centering
    \begin{tikzpicture}

    \begin{groupplot}[
        group style={
            group name=my plots,
            group size=2 by 1, 
            horizontal sep=5pt
        }, 
        width=5.0cm, 
        height=5.0cm, 
        grid=both,
        minor x tick num=3,
        grid style={line width=.1pt, draw=gray!20},
        major grid style={line width=.1pt, draw=gray!60},
        ytick distance=10^1,
        ytick={0.0001, 0.001, 0.01, 0.1, 1.0, 10.0, 100},
        yticklabels={,$10^{-3}$,,$10^{-1}$,,$10^{1}$},
        ymin=0.0003, ymax=50,
        ylabel shift=-0.25cm,
    ]
    \nextgroupplot[
        title={Dahlquist},
        xmode=log,
        ymode=log,
        xlabel={ODE States $N$},
        ylabel={Runtime [s]},
        xtick align=inside,
        ytick align=inside,
        legend entries={Parallel Newton;, Parareal;, Sequential},
        legend to name=implicit_legend,
        legend style={legend columns=-1},
    ]

        \addplot [
            semithick, 
            black, 
            mark=o, 
            mark size=2, 
            mark options={solid}
        ] table {%
            40 0.00041704
            400 0.00044188
            4001 0.0005307
            40000 0.00081491
        };
         \addplot [
            semithick, 
            black, 
            mark=triangle, 
            mark size=2, 
            mark options={solid}
        ] table {%
            40 0.0115128
            400 0.04162772
            4001 0.09763474
            40000 0.22486305
        };
        \addplot [
            semithick, 
            black, 
            mark=square, 
            mark size=2, 
            mark options={solid}
        ] table {%
            40 0.00541027
            400 0.0398442
            4001 0.37844076
            40000 3.21991856
        };

    \nextgroupplot[
        title={Robertson},
        xmode=log,
        ymode=log,
        xlabel={ODE States $N$},
        ylabel={},
        yticklabel={\empty},
        xtick align=inside,
        ytick align=inside,
        ymin=0.0003, ymax=50,
    ]
        \addplot [
            semithick, 
            black, 
            mark=o, 
            mark size=2, 
            mark options={solid}
        ] table {%
        5000 0.00684168
        50000 0.02231853
        100000 0.03630676
        };
        \addplot [
            semithick, 
            black, 
            mark=triangle, 
            mark size=2, 
            mark options={solid}
        ] table {%
        5000 0.43008983
        50000 1.05268574
        100000 1.46933091
        };
        \addplot [
            semithick, 
            black, 
            mark=square, 
            mark size=2, 
            mark options={solid}
        ] table {%
        5000 1.05407581
        50000 7.27644947
        100000 13.69847407
        };
        
\end{groupplot}
\node[anchor=north] (title-x) at ($(my plots c1r1.south east)!0.5!(my plots c2r1.south west)-(0,1.0cm)$) {\ref{implicit_legend}};
\end{tikzpicture}
    \caption{Average runtime comparison of implicit parallel Newton-based method and the sequential method for the Dahlquist test problem and the Robertson chemical reaction system over a fixed time interval with varying numbers of intermediary steps.}
    \label{fig:implicit-runtime}
\end{figure}
\begin{figure}[t]
    \centering
    \begin{tikzpicture}

    \begin{groupplot}[
        group style={
            group name=my plots,
            group size=2 by 1, 
            horizontal sep=5pt
        }, 
        width=5.0cm, 
        height=5.0cm, 
        grid=both,
        minor x tick num=3,
        grid style={line width=.1pt, draw=gray!20},
        major grid style={line width=.1pt, draw=gray!60},
           ytick distance=10^1,
        ytick={1e-20, 1e-15, 1e-10, 1e-5, 1},
        ymin=1e-17, ymax=5,
        ylabel shift=-0.25cm,
    ]
    \nextgroupplot[
        title={Dahlquist},
        ymode=log,
        xlabel={Iteration $k$},
        ylabel={},
        xtick align=inside,
        ytick align=inside,
        legend entries={$\left\|h\left(\xi^{(k)}\right)\right\|_\infty$;, $\left\|R^{(k)}\right\|_\infty$},
        legend to name=imp_residual_legend,
        legend style={legend columns=-1},
    ]
        \addplot [
            semithick, 
            black, 
            mark=o, 
            mark size=2, 
            mark options={solid}
        ] table {%
            0 0.9090909090909091
            1 1.0092936587501423e-16
        };
        \addplot [
            semithick, 
            black, 
            mark=triangle, 
            mark size=2, 
            mark options={solid}
        ] table {%
            0 0.06052606994622256
            1 0.004579256428918821
            2 0.0005413367088250645
            3 5.2423973604484586e-05
            4 2.538413674584017e-06
            5 4.9164865137396596e-08
            6 1.1102230246251565e-16
        };
    \nextgroupplot[
        title={Robertson},
        ymode=log,
        xlabel={Iteration $k$},
        ylabel={},
        yticklabel={\empty},
        xtick align=inside,
        ytick align=inside,
        ytick distance=10^1,
        ylabel shift=-0.25cm,
        ytick={1e-20, 1e-15, 1e-10, 1e-5, 1},
    ]
        \addplot [
            semithick, 
            black, 
            mark=o, 
            mark size=2, 
            mark options={solid}
        ] table {%
0 0.9960159362549801
1 0.49999999884389346
2 0.24999999929574218
3 0.12499999939289713
4 0.06249999917647283
5 0.031249998508849103
6 0.015624996942027056
7 0.007812493261799104
8 0.0039062340895818053
9 0.0028824718016214335
10 0.003000436642158641
11 0.0030384225947932206
12 0.002974022204255907
13 0.002787207257987855
14 0.0024564428483596326
15 0.0021788123318744277
16 0.0016513051380945202
17 0.0008355377048871598
18 0.0002960337916729992
19 7.408455095058445e-05
20 7.968501949473762e-06
21 4.75144352394775e-08
22 2.7081222852029127e-11
23 1.1096994180677665e-16
        };

        \addplot [
            semithick, 
            black, 
            mark=triangle, 
            mark size=2, 
            mark options={solid}
        ] table {%
0 0.03146237230681484
1 0.0038577423309130443
2 0.00021637567305596295
3 8.828612217071097e-06
4 2.8780337713030235e-07
5 7.553947578564646e-09
6 1.61618274319153e-10
7 2.883748795312613e-12
8 4.374278717023117e-14
9 1.6653345369377348e-15
10 3.3306690738754696e-16
        };
\end{groupplot}
\node at ($(current bounding box.south)+(0,-0.5cm)$) {\ref{imp_residual_legend}};
\end{tikzpicture}
    \caption{Evolution of the residual norms over Newton and Parareal iterations for the Dahlquist test and Robertson's chemical reaction system.}
    \label{fig:implicit-residuals}
\end{figure}
\begin{figure}
    \centering
    \begin{tikzpicture}

    \begin{groupplot}[
        group style={
            group name=my plots,
            group size=3 by 1, 
            horizontal sep=5pt
        }, 
        width=5.0cm, 
        height=5.0cm, 
        grid=both,
        minor x tick num=3,
        grid style={line width=.1pt, draw=gray!20},
        major grid style={line width=.1pt, draw=gray!60},
        ylabel shift=-0.25cm,
    ]
    \nextgroupplot[
        title={Dahlquist},
        xmode=log,
        xlabel={$\delta t$},
        ylabel={\empty},
        xtick align=inside,
        ytick align=inside,
        ytick={2, 4, 6, 8, 10, 12},
        yticklabels={2, , 6, , 10, },
        ymin=1, ymax=13,
        legend entries={Parallel Newton $K$;,Parareal $K$},
        legend to name=implicit_iterations_legend,
        legend style={legend columns=-1},
    ]
        \addplot [
            semithick, 
            black, 
            mark=o, 
            mark size=2, 
            mark options={solid}
        ] table {%
           0.1 2
           0.01 2
           0.001 2
           0.0001 2
        };
        \addplot [
            semithick, 
            black, 
            mark=triangle, 
            mark size=2, 
            mark options={solid}
        ] table {%
           0.1 7
           0.01 9
           0.001 7
           0.0001 5
        };
    \nextgroupplot[
        title={Robertson},
        xmode=log,
        xlabel={$\delta t$},
        yticklabel={\empty},
        xtick align=inside,
        ytick align=inside,
        yticklabel pos=right,
        ytick pos=right,
        ytick={8, 10, 12, 14, 16, 18, 20, 22, 24},
        yticklabels={8, , , 14, , , 20, , 24},
        ymin=7, ymax=25,
    ]
        \addplot [
            semithick, 
            black, 
            mark=o, 
            mark size=2, 
            mark options={solid}
        ] table {%
           0.1 24
           0.01 24
           0.005 24
        };
        \addplot [
            semithick, 
            black, 
            mark=triangle, 
            mark size=2, 
            mark options={solid}
        ] table {%
            0.1 11
           0.01 10
           0.005 10
        };

\end{groupplot}
\node at ($(current bounding box.south)+(0,-0.5cm)$) {\pgfplotslegendfromname{implicit_iterations_legend}};
\end{tikzpicture}
    \caption{Influence of the time-step $\delta t$ on the convergence iteration number $K$ for the Dahlquist test, and the Robertson system.}
    \label{fig:implicit-iterations}
\end{figure}

We analyze the convergence speed of Newton's method using numerical results. In Figure \ref{fig:implicit-residuals} we plot the infinity norm of the residuals at each iteration obtained from integrating \eqref{eq:dahlquist} and \eqref{eq:robertson} with $\delta t=10^{-1}$. The initial guess for Newton's method remains the same, a vector of zeros, for both cases. Furthermore, we compare the error in Newton's method against the error in Parareal. Newton's method reaches a high solution accuracy after only two iterations compared to Parareal for the Dahlquist test. In contrast, Parareal converges considerably faster than Newton for the Robertson chemical reaction problem. This behaviour is again expected since Newton's method is highly dependent on the initial guess. A vector of zeros is close to the solution of the Dahlquist test problem; hence, the small number of Newton iterations. Nevertheless, in all cases Newton's method exhibits less computational time due to the prefix scan implementation compared to Parareal.

Similar to the explicit case, we finalize the numerical experiments by investigating how the variation of $\delta t$ influences the convergence speed. The resulting plots are displayed in Figure \ref{fig:implicit-iterations}. 

\section{Conclusion}
We presented a computationally fast method for the temporal parallelization of nonlinear ODE solvers. Our approach reformulates the numerical approximation of the ODE solution as a root-finding problem, which we solve iteratively using Newton's method. We introduce a parallel strategy for computing the Newton step, leveraging the power of massively parallel hardware. Our strategy's computational time scales logarithmically with the number of discretization steps. To demonstrate the efficiency of our method, we deploy it on a GPU and apply it to a series of benchmark problems. Furthermore, we compare our method against the sequential approach and the widely used Parareal method.

Future work could investigate the application of the presented method for partial differential equations, multiscale problems such as weather prediction systems, as well as simulation of stochastic differential equations.

\appendix
\section{Convergence proof of the parallel Newton's step for explicit ODE solvers}\label{ap:convergence_explicit_ODEs}
From the definition of the Newton step\\  $\Delta \xi^{(k)}=-H^{-1}\left(\xi^{(k)}\right)h\left(\xi^{(k)}\right)$ and the root-finding problem $h\left(\xi^*\right) = 0$ we have
\begin{equation}
\begin{split}
    \xi^{(k)} + \Delta \xi^{(k)} - \xi^* &= \xi^{(k)} - \xi^* - H^{-1}\left(\xi^{(k)}\right) h\left(\xi^{(k)}\right)\\&=H^{-1}\left(\xi^{(k)}\right)\left[H\left(\xi^{(k)}\right)\left(\xi^{(k)} - \xi^*\right) - \left(h\left(\xi^{(k)}\right) - h\left(\xi^*\right)\right)\right]. 
\end{split}\label{err}
\end{equation}

It follows from Taylor's theorem that
\begin{equation}
    \begin{aligned}
    h(\xi^{*})&=h\left(\xi^{(k)}+\left(\xi^{*}-\xi^{(k)}\right)\right)
    \\&=h\left(\xi^{(k)}\right)+H\left(\xi^{(k)}+\alpha\left(\xi^{*}-\xi^{(k)}\right)\right)\left(\xi^{*}-\xi^{(k)}\right).
    \end{aligned}
\end{equation}
where $\alpha \in [0,1]$. The term $ h\left(\xi^{*}\right)-h\left(\xi^{(k)}\right)$ is then given by
\begin{equation}\label{aux}
    h\left(\xi^{(k)}\right) - h\left(\xi^*\right)= H\left(\xi^{(k)}+\alpha\left(\xi^{*}-\xi^{(k)}\right)\right)\left(\xi^{*}-\xi^{(k)
    }\right).
\end{equation}
Substituting \eqref{aux} to the last term in \eqref{err} we obtain
\begin{equation}
\begin{split}
    &\left\|H\left(\xi^{(k)}\right)\left(\xi^{(k)} - \xi^*\right) - \left(h\left(\xi^{(k)}\right) - h\left(\xi^*\right)\right)\right\|\\ & \qquad= \left\|H\left(\xi^{(k)}\right)\left(\xi^{(k)}-\xi^{*}\right) -H\left(\xi^{(k)}+\alpha\left(\xi^{*}-\xi^{(k)}\right)\right)\left(\xi^{*}-\xi^{(k)}\right)\right\|\\
     & \qquad= \left\|\left[H\left(\xi^{(k)}\right) -H\left(\xi^{(k)}+\alpha\left(\xi^{*}-\xi^{(k)}\right)\right)\right]\left(\xi^{*}-\xi^{(k)}\right)\right\|.
    \end{split}
\end{equation}
From the submultiplicative property of matrix norms, we get
\begin{equation}
\begin{split}
    &\left\|H\left(\xi^{(k)}\right)\left(\xi^{(k)} - \xi^*\right) - \left(h\left(\xi^{(k)}\right) - h\left(\xi^*\right)\right)\right\|\\
    & \qquad \leq \left\|\left(\xi^{*}-\xi^{(k)}\right)\right\| \left\|\left[H\left(\xi^{(k)}\right) -H\left(\xi^{(k)}+\alpha\left(\xi^{*}-\xi^{(k)}\right)\right)\right]\right\|,
    \end{split}
\end{equation}
and following Assumption \ref{assum2} leads us to the upper bound
\begin{equation}\label{onetermbound}
\begin{split}
    \left\|H\left(\xi^{(k)}\right)\left(\xi^{(k)} - \xi^*\right) - \left(h\left(\xi^{(k)}\right) - h\left(\xi^*\right)\right)\right\| &\leq \left\|\left(\xi^{*}-\xi^{(k)}\right)\right\| \left\|\alpha L \left(\xi^{(k)}-\xi^{*}\right)\right\|\\
    & \leq L \left \|(\xi^{*}-\xi^{(k)})\right\|^2.
    \end{split}
\end{equation}
Hence, \eqref{err} is bounded by \eqref{onetermbound} as follows
\begin{equation}\label{convergence}
\begin{split}
    \left\|\xi^{(k)} + \Delta \xi^{(k)} - \xi^*\right\|  &= \left\|H^{-1}\left(\xi^{(k)}\right)\left[H\left(\xi^{(k)}\right)\left(\xi^{(k)} - \xi^*\right) - \left(h\left(\xi^{(k)}\right) - h\left(\xi^*\right)\right)\right]\right\|\\
             & \leq L \left\|H^{-1}\left(\xi^{(k)}\right)\right\|\left \|\left(\xi^{*}-\xi^{(k)}\right)\right\|^2.
\end{split}
\end{equation}
Next, we prove that $\left\| H^{-1}\left(\xi ^{(k)}\right)\right\|$ is bounded. Let us find an upper bound in terms of $H(\xi^{*})$. Recall that $H^{-1}(\xi^{*})$ exists according to \eqref{eq:jacobian} and it is a lower bidiagonal block matrix with identity blocks on the diagonal, hence, it is always invertible. Therefore, we can write
\begin{equation}
    \left\|H^{-1}\left(\xi^{(k)}\right)\right\|=\left\|\left[H\left(\xi^{(k)}\right)+H(\xi^{*})-H(\xi^{*})\right]^{-1}\right\|.
\end{equation}
The right-hand side term in the norm can be rewritten as 
\begin{equation}
\begin{split}
    &\left(H\left(\xi^{(k)}\right)+H(\xi^{*})-H(\xi^{*})\right)^{-1}\\
    &=H^{-1}(\xi^{*})\left[I+H^{-1}(\xi^{*})\left(H\left(\xi^{(k)}\right)-H(\xi^{*})\right)\right]^{-1}\\
    &=H^{-1}(\xi^{*})\left[I-\left(I-H^{-1}(\xi^{*})H\left(\xi^{(k)}\right)\right)\right]^{-1}.
\end{split}\label{eq:rhs-norm}
\end{equation} 
Consider a ball-neighborhood $\left\|\xi^{(k)}-\xi^{*}\right\| \leq r$ and fix $r=\frac{1}{2L\|H^{-1}(\xi^{*})\|}$. Then by Assumption \ref{assum2}, we have
\begin{equation}\label{ballneighbour}
    \|H(\xi)-H(\xi^{*})\| \leq L \|\xi-\xi^{*}\| \leq Lr.
\end{equation}
Replacing $r$ in \eqref{ballneighbour} yields the following result 
\begin{equation}\label{eq:ball-result}
\begin{aligned}
    \left\|I-H^{-1}\left(\xi^{*}\right)H\left(\xi^{(k)}\right)\right\|&=\left\|H^{-1}(\xi^{*})\left(H(\xi^{*})-H\left(\xi^{(k)}\right)\right)\right\| \\
    &\leq \left\|H^{-1}(\xi^{*})\right\|\left\|H(\xi^{*})-H\left(\xi^{(k)}\right)\right\|\\
    &\leq L \left\|H^{-1}(\xi^{*})\right\| \|\xi-\xi^{*}\|\\
    &\leq Lr \left\|H^{-1}(\xi^{*})\right\|=\frac{1}{2} \leq 1.
    \end{aligned}
\end{equation}
The result in \eqref{eq:ball-result} implies that we can write \eqref{eq:rhs-norm} as a Neumann series 
\begin{equation}
 \begin{aligned}
    \left(H\left(\xi^{(k)}\right)+H(\xi^{*})-H(\xi^{*})\right)^{-1}&= H^{-1}(\xi^{*})\sum_{n=0}^{\infty} \left(I-H^{-1}(\xi^{*})H\left(\xi^{(k)}\right)\right)^n.
    \end{aligned}
\end{equation}
Therefore, 
\begin{equation}
    \begin{split}
       &\left\| \left(H\left(\xi^{(k)}\right)+H(\xi^{*})-H(\xi^{*})\right)^{-1}\right\|\\
       &=\left\|H^{-1}(\xi^{*})\sum_{n=0}^{\infty} \left(I-H^{-1}(\xi^{*})H\left(\xi^{(k)}\right)\right)^n\right\|\\
       &\leq \left\|H^{-1}(\xi^{*})\right\|\left\|\sum_{n=0}^{\infty} \left(I-H^{-1}(\xi^{*})H\left(\xi^{(k)}\right)\right)^n \right\|\\
       &\leq \left\|H^{-1}(\xi^{*})\right\|\sum_{n=0}^{\infty} \left\|\left(I-H^{-1}(\xi^{*})H\left(\xi^{(k)}\right)\right)^n \right\|\\
       &= \left\|H^{-1}(\xi^{*})\right\|\sum_{n=0}^{\infty} \left\|\left(H^{-1}(\xi^{*})\left[H(\xi^{*})-H\left(\xi^{(k)}\right)\right]\right)^n \right\|\\
       &= \left\|H^{-1}(\xi^{*})\right\|\sum_{n=0}^{\infty} \left(\left\|H^{-1}(\xi^{*})\right\|\left\|H(\xi^{*})-H\left(\xi^{(k)}\right)\right\|\right)^n,
    \end{split}
\end{equation}
which, again, using the property of series yields
\begin{equation}\label{gseries}
    \begin{aligned}
         \left\| H^{-1}\left(\xi^{(k)}\right)\right\| \leq \frac{\|H^{-1}(\xi^{*})\|}{1-\left\|H^{-1}(\xi^{*})\right\|\left\|H(\xi^{*})-H\left(\xi^{(k)}\right)\right\|}.
    \end{aligned}
\end{equation}
Subsequently, using \eqref{gseries} and \eqref{ballneighbour} for the fixed $r=\frac{1}{2L\|H^{-1}(\xi^{*})\|}$, we get
\begin{equation}\label{boundofH}
     \left\| H^{-1}\left(\xi^{(k)}\right)\right\| \leq 2 \left\| H^{-1}(\xi^{*})\right\|.
\end{equation}
From \eqref{boundofH} and \eqref{convergence} follows that
\begin{equation}\label{convergence2}
\begin{split}
    \left\|\xi^{(k)} + \Delta \xi^{(k)} - \xi^*\right\|  \leq 2 L \left\|H^{-1}\left(\xi^{*}\right)\right\|\left \|\left(\xi^{*}-\xi^{(k)}\right)\right\|^2,
\end{split}
\end{equation}
which proves the quadratic convergence rate of the Newton iterates.

Next, we prove the quadratic convergence of the residual norms to zero. We start by expanding $\|h\left(\xi^{(k+1)}\right)\|$ as
\begin{equation}
\begin{split}
    &\left\|h\left(\xi^{(k+1)}\right)\right\|\\
    &=\left\|h\left(\xi^{(k+1)}\right)-h\left(\xi^{(k)}\right)-H\left(\xi^{(k)}\right)\left(\xi^{(k+1)}-\xi^{(k)}\right)\right\|\\
    &=\left\|H\left(\xi^{(k)}+\alpha \left(\xi^{(k+1)}-\xi^{(k)}\right)\right)\left(\xi^{(k+1)}-\xi^{(k)}\right)-H\left(\xi^{(k)}\right)\left(\xi^{(k+1)}-\xi^{(k)}\right)\right\|\\
    &=\left\|\left[H\left(\xi^{(k)}+\alpha \left(\xi^{(k+1)}-\xi^{(k)}\right)\right)-H\left(\xi^{(k)}\right)\right]\left(\xi^{(k+1)}-\xi^{(k)}\right)\right\|\\
    &\leq\left\|H\left(\xi^{(k)}+\alpha \left(\xi^{(k+1)}-\xi^{(k)}\right)\right)-H\left(\xi^{(k)}\right)\right\|\left\|\left(\xi^{(k+1)}-\xi^{(k)}\right)\right\|.
\end{split}    
\end{equation}
Invoking Assumption \ref{assum2} leads us to
\begin{equation}
    \begin{aligned}
        \left\|h\left(\xi^{(k+1)}\right)\right\|& \leq\left\|\left(H\left(\xi^{(k)}+\alpha \left(\xi^{(k+1)}-\xi^{(k)}\right)\right)-H\left(\xi^{(k)}\right)\right)\right\|\left\|\left(\xi^{(k+1)}-\xi^{(k)}\right)\right\| \\
        &\leq L \alpha\left\|\xi^{(k+1)}-\xi^{(k)}\right\|^2 \leq L \left\|\xi^{(k+1)}-\xi^{(k)}\right\|^2.\\
    \end{aligned}
\end{equation}
We also know that $\Delta \xi^{(k)}=\xi^{(k+1)}-\xi^{(k)}$. Therefore,
\begin{equation}
     \left\|h\left(\xi^{(k+1)}\right)\right\| \leq  L \left\|\Delta \xi^{(k)}\right\|^2.\\
\end{equation}
Using  $h\left(\xi^{k+1}\right)+H\left(\xi^{(k)}\right)\Delta \xi^{(k)}=0$, gives us the following:
\begin{equation}
     \left\|h\left(\xi^{(k+1)}\right)\right\| \leq  L\ \left\|H^{-1}\left(\xi^{(k)}\right)h\left(\xi^{(k)}\right)\right\|^2 
     \leq L \left\|H^{-1}\left(\xi^{(k)}\right)\right\|^2\left\|h\left(\xi^{(k)}\right)\right\|^2.
\end{equation}
Finally, using the inequlity \eqref{boundofH} gives
\begin{equation}
     \left\|h\left(\xi^{(k+1)}\right)\right\|   
     \leq 4L \left\|H^{-1}\left(\xi^{*}\right)\right\|^2\left\|h\left(\xi^{(k)}\right)\right\|^2.
\end{equation}
\bibliographystyle{siamplain}
\bibliography{parallel-ode-bibliography}

\begin{thebibliography}{10}

\bibitem{bernstein2009matrix}
{\sc D.~S. Bernstein}, {\em Matrix Mathematics: Theory, Facts, and Formulas}, Princeton University Press, 2nd~ed., 2009.

\bibitem{bhatt2025introducing}
{\sc R.~Bhatt, L.~Debreu, and A.~Vidard}, {\em Introducing time parallelization within data assimilation}, SIAM Journal on Scientific Computing, 47 (2025), pp.~B533--B557.

\bibitem{blelloch1989scans}
{\sc G.~E. Blelloch}, {\em Scans as primitive parallel operations}, IEEE Transactions on Computers, 38 (1989), pp.~1526--1538.

\bibitem{blelloch1990prefix}
{\sc G.~E. Blelloch}, {\em Prefix sums and their applications}, tech. report, School of Computer Science, Carnegie Mellon University Pittsburgh, PA, USA, 1990.

\bibitem{bosch2024parallel}
{\sc N.~Bosch, A.~Corenflos, F.~Yaghoobi, F.~Tronarp, P.~Hennig, and S.~S{\"a}rkk{\"a}}, {\em Parallel-in-time probabilistic numerical {ODE} solvers}, Journal of Machine Learning Research, 25 (2024), pp.~1--27.

\bibitem{boyd2001chebyshev}
{\sc J.~P. Boyd}, {\em Chebyshev and Fourier Spectral Methods}, Courier Corporation, 2001.

\bibitem{jax2018github}
{\sc J.~Bradbury, R.~Frostig, P.~Hawkins, M.~J. Johnson, C.~Leary, D.~Maclaurin, G.~Necula, A.~Paszke, J.~Vander{P}las, S.~Wanderman-{M}ilne, and Q.~Zhang}, {\em {JAX}: composable transformations of {P}ython+{N}um{P}y programs}, 2018, \url{http://github.com/jax-ml/jax}.

\bibitem{buonomo1998periodic}
{\sc A.~Buonomo}, {\em The periodic solution of van der {P}ol's equation}, SIAM Journal on Applied Mathematics, 59 (1998), pp.~156--171.

\bibitem{butcher2016numerical}
{\sc J.~C. Butcher}, {\em Numerical Methods for Ordinary Differential Equations}, John Wiley \& Sons, 2016.

\bibitem{corless2019optimal}
{\sc R.~M. Corless, C.~Y. Kaya, and R.~H. Moir}, {\em Optimal residuals and the {D}ahlquist test problem}, Numerical Algorithms, 81 (2019), pp.~1253--1274.

\bibitem{danieli2023deeppcr}
{\sc F.~Danieli, M.~Sarabia, X.~Suau~Cuadros, P.~Rodriguez, and L.~Zappella}, {\em Deep{PCR}: Parallelizing sequential operations in neural networks}, Advances in Neural Information Processing Systems, 36 (2023), pp.~47598--47625.

\bibitem{dobrev2017two}
{\sc V.~A. Dobrev, T.~Kolev, N.~A. Petersson, and J.~B. Schroder}, {\em Two-level convergence theory for multigrid reduction in time ({MGRIT})}, SIAM Journal on Scientific Computing, 39 (2017), pp.~S501--S527.

\bibitem{falgout2014parallel}
{\sc R.~D. Falgout, S.~Friedhoff, T.~V. Kolev, S.~P. MacLachlan, and J.~B. Schroder}, {\em Parallel time integration with multigrid}, SIAM Journal on Scientific Computing, 36 (2014), pp.~C635--C661.

\bibitem{fang2022parallel}
{\sc L.~Fang, S.~Vandewalle, and J.~Meyers}, {\em A parallel-in-time multiple shooting algorithm for large-scale {PDE}-constrained optimal control problems}, Journal of Computational Physics, 452 (2022), p.~110926.

\bibitem{gander201550}
{\sc M.~J. Gander}, {\em 50 years of time parallel time integration}, in Multiple Shooting and Time Domain Decomposition Methods: MuS-TDD, Heidelberg, May 6-8, 2013, Springer, 2015, pp.~69--113.

\bibitem{gander2013paraexp}
{\sc M.~J. Gander and S.~Gu\"uttel}, {\em {PARAEXP}: {A} parallel integrator for linear initial-value problems}, SIAM Journal on Scientific Computing, 35 (2013), pp.~C123--C142.

\bibitem{gander2020paraopt}
{\sc M.~J. Gander, F.~Kwok, and J.~Salomon}, {\em {PARAOPT}: A parareal algorithm for optimality systems}, SIAM Journal on Scientific Computing, 42 (2020), pp.~A2773--A2802.

\bibitem{gander2024time}
{\sc M.~J. Gander and T.~Lunet}, {\em Time parallel time integration}, SIAM, 2024.

\bibitem{gander2025parareal}
{\sc M.~J. Gander, M.~Ohlberger, and S.~Rave}, {\em A parareal algorithm with spectral coarse solver}, arXiv preprint arXiv:2508.08873,  (2025).

\bibitem{gander2007analysis}
{\sc M.~J. Gander and S.~Vandewalle}, {\em Analysis of the parareal time-parallel time-integration method}, SIAM Journal on Scientific Computing, 29 (2007), pp.~556--578.

\bibitem{gattiglio2024randnet}
{\sc G.~Gattiglio, L.~Grigoryeva, and M.~Tamborrino}, {\em Rand{N}et-{P}arareal: a time-parallel {PDE} solver using random neural networks}, Advances in Neural Information Processing Systems, 37 (2024), pp.~94993--95025.

\bibitem{gattiglio2025prob}
{\sc G.~Gattiglio, L.~Grigoryeva, and M.~Tamborrino}, {\em Prob-{GP}arareal: A probabilistic numerical parallel-in-time solver for differential equations}, arXiv preprint arXiv:2509.03945,  (2025).

\bibitem{gonzalez2026predictability}
{\sc X.~Gonzalez, L.~Kozachkov, D.~Zoltowski, K.~Clarkson, and S.~Linderman}, {\em Predictability enables parallelization of nonlinear state space models}, Advances in Neural Information Processing Systems, 38 (2026), pp.~19101--19147.

\bibitem{gonzalez2024towards}
{\sc X.~Gonzalez, A.~Warrington, J.~T. Smith, and S.~W. Linderman}, {\em Towards scalable and stable parallelization of nonlinear rnns}, Advances in Neural Information Processing Systems, 37 (2024), pp.~5817--5849.

\bibitem{hairer1996solving}
{\sc E.~Hairer and G.~Wanner}, {\em Solving Ordinary Differential Equations II: Stiff and Differential-Algebraic Problems}, Springer, 2nd~ed., 1996.

\bibitem{hairer1993solving}
{\sc E.~Hairer, G.~Wanner, and S.~P. N{\o}rsett}, {\em Solving Ordinary Differential Equations I: Nonstiff Problems}, Springer, 2nd~ed., 1993.

\bibitem{haut2014asymptotic}
{\sc T.~Haut and B.~Wingate}, {\em An asymptotic parallel-in-time method for highly oscillatory {PDE}s}, SIAM Journal on Scientific Computing, 36 (2014), pp.~A693--A713.

\bibitem{iacob2025parallel}
{\sc C.~Iacob, H.~Abdulsamad, and S.~S{\"a}rkk{\"a}}, {\em A parallel-in-time {N}ewton’s method for nonlinear model predictive control}, IEEE Transactions on Control Systems Technology,  (2025).

\bibitem{iqbal2024parallel}
{\sc S.~Iqbal, H.~Abdulsamad, T.~Cator, U.~Braga-Neto, and S.~S{\"a}rkk{\"a}}, {\em Parallel-in-time probabilistic solutions for time-dependent nonlinear partial differential equations}, in 2024 IEEE 34th International Workshop on Machine Learning for Signal Processing (MLSP), 2024, pp.~1--6.

\bibitem{jin2025optimizing}
{\sc B.~Jin, Q.~Lin, and Z.~Zhou}, {\em Optimizing coarse propagators in parareal algorithms}, SIAM Journal on Scientific Computing, 47 (2025), pp.~A735--A761.

\bibitem{legoll2020parareal}
{\sc F.~Legoll, T.~Leli{\`e}vre, K.~Myerscough, and G.~Samaey}, {\em Parareal computation of stochastic differential equations with time-scale separation: a numerical convergence study}, Computing and Visualization in Science, 23 (2020), p.~9.

\bibitem{lewis2012optimal}
{\sc F.~L. Lewis, D.~Vrabie, and V.~L. Syrmos}, {\em Optimal Control}, John Wiley \& Sons, 3rd~ed., 2012.

\bibitem{lim2024parallelizing}
{\sc Y.~H. Lim, Q.~Zhu, J.~Selfridge, and M.~Firmansyah}, {\em Parallelizing non-linear sequential models over the sequence length}, in International Conference on Learning Representations, vol.~2024, 2024, pp.~55334--55360.

\bibitem{lions2001resolution}
{\sc J.-L. Lions, Y.~Maday, and G.~Turinici}, {\em R{\'e}solution d'{EDP} par un sch{\'e}ma en temps ``parar{\'e}el''}, Comptes Rendus de l'Acad{\'e}mie des Sciences-Series I-Mathematics, 332 (2001), pp.~661--668.

\bibitem{lorenz2017deterministic}
{\sc E.~N. Lorenz}, {\em Deterministic nonperiodic flow}, in Universality in Chaos, 2nd edition, Routledge, 2017, pp.~367--378.

\bibitem{maday2013parareal}
{\sc Y.~Maday, M.-K. Riahi, and J.~Salomon}, {\em Parareal in time intermediate targets methods for optimal control problems}, in Control and optimization with PDE constraints, Springer, 2013, pp.~79--92.

\bibitem{mathew2010analysis}
{\sc T.~P. Mathew, M.~Sarkis, and C.~E. Schaerer}, {\em Analysis of block parareal preconditioners for parabolic optimal control problems}, SIAM Journal on Scientific Computing, 32 (2010), pp.~1180--1200.

\bibitem{nocedal2006numerical}
{\sc J.~Nocedal and S.~J. Wright}, {\em Numerical Optimization}, Springer, 2006.

\bibitem{pentland2023error}
{\sc K.~Pentland, M.~Tamborrino, and T.~J. Sullivan}, {\em Error bound analysis of the stochastic parareal algorithm}, SIAM Journal on Scientific Computing, 45 (2023), pp.~A2657--A2678.

\bibitem{rader2024optimistix}
{\sc J.~Rader, T.~Lyons, and P.~Kidger}, {\em Optimistix: modular optimisation in {JAX} and {E}quinox}, arXiv preprint arXiv:2402.09983,  (2024).

\bibitem{robertson1966reaction}
{\sc H.~H. Robertson}, {\em The solution of a set of reaction rate equations}, in Numerical Analysis: An Introduction, J.~Walsh, ed., Academic Press, London, England, 1966, pp.~178--182.

\bibitem{rosemeier2024multilevel}
{\sc J.~Rosemeier, T.~Haut, and B.~Wingate}, {\em Multilevel parareal algorithm with averaging for oscillatory problems}, SIAM Journal on Scientific Computing, 46 (2024), pp.~A2709--A2736.

\bibitem{sarkka2022temporal}
{\sc S.~S{\"a}rkk{\"a} and {\'A}.~F. Garc{\'\i}a-Fern{\'a}ndez}, {\em Temporal parallelization of dynamic programming and linear quadratic control}, IEEE Transactions on Automatic Control, 68 (2022), pp.~851--866.

\bibitem{sarkka2024temporal}
{\sc S.~S{\"a}rkk{\"a} and {\'A}.~F. Garc{\'\i}a-Fern{\'a}ndez}, {\em Temporal parallelisation of the {HJB} equation and continuous-time linear quadratic control}, IEEE Transactions on Automatic Control,  (2024).

\bibitem{sarkka2023bayesian}
{\sc S.~S{\"a}rkk{\"a} and L.~Svensson}, {\em Bayesian Filtering and Smoothing}, vol.~17, Cambridge University Press, 2023.

\bibitem{satish2009designing}
{\sc N.~Satish, M.~Harris, and M.~Garland}, {\em Designing efficient sorting algorithms for manycore {GPU}s}, in 2009 IEEE International Symposium on Parallel \& Distributed Processing, IEEE, 2009, pp.~1--10.

\bibitem{sutton1998reinforcement}
{\sc R.~S. Sutton and A.~G. Barto}, {\em Reinforcement Learning: An Introduction}, vol.~1, MIT Press Cambridge, 1998.

\bibitem{underactuated}
{\sc R.~Tedrake}, {\em Underactuated robotics}.
\newblock Course Notes for MIT 6.832, 2023, \url{https://underactuated.csail.mit.edu}.

\bibitem{tronarp2019probabilistic}
{\sc F.~Tronarp, H.~Kersting, S.~S{\"a}rkk{\"a}, and P.~Hennig}, {\em Probabilistic solutions to ordinary differential equations as nonlinear {B}ayesian filtering: a new perspective}, Statistics and Computing, 29 (2019), pp.~1297--1315.

\bibitem{yaghoobi2025parallel}
{\sc F.~Yaghoobi, A.~Corenflos, S.~Hassan, and S.~S{\"a}rkk{\"a}}, {\em Parallel square-root statistical linear regression for inference in nonlinear state space models}, SIAM Journal on Scientific Computing, 47 (2025), pp.~B454--B476.

\bibitem{yaghoobi2024parallel}
{\sc F.~Yaghoobi and S.~S{\"a}rkk{\"a}}, {\em Parallel state estimation for systems with integrated measurements}, IEEE Signal Processing Letters,  (2024).

\bibitem{yang2017parallel}
{\sc Y.~Yang, Y.~Wu, and J.~Pan}, {\em Parallel dynamics computation using prefix sum operations}, IEEE Robotics and Automation Letters, 2 (2017), pp.~1296--1303.

\bibitem{zoltowski2026parallelizing}
{\sc D.~Zoltowski, S.~Wu, X.~Gonzalez, L.~Kozachkov, and S.~Linderman}, {\em Parallelizing {MCMC} across the sequence length}, Advances in Neural Information Processing Systems, 38 (2026), pp.~22242--22277.

\end{thebibliography}
\end{document}